\documentclass[twoside,centertags]{amsart}
\usepackage{amsmath,amsthm,amssymb,latexsym,graphicx}
\usepackage{xypic}

\newtheorem{theorem}{Theorem} [subsection]
\newtheorem{lemma}[theorem]{Lemma}
\newtheorem{corollary}[theorem]{Corollary}
\newtheorem{proposition}[theorem]{Proposition}
\theoremstyle{definition}
\newtheorem{definition}[theorem]{Definition}

\newtheorem{remark}[theorem]{Remark}

\newtheorem{example}[theorem]{Example}

\numberwithin{equation}{section}

\newcommand{\Z}{\mathbb{Z}}
\newcommand{\R}{\mathbb{R}}

\newcommand{\ho}{\mathrm{hocolim}}
\newcommand{\Ob}{\mathrm{Ob}}
\newcommand{\Mor}{\mathrm{Mor}}
\newcommand{\hof}{\mathrm{hofiber}}
\newcommand{\h}{\mathrm{hocolim_{\mathcal{I}^{op}}}}
\newcommand{\co}{\mathrm{colim}}
\newcommand{\End}{\mathrm{End}}
\newcommand{\Diff}{\mathrm{Diff}}

\renewcommand{\leq}{\leqslant}
\renewcommand{\geq}{\geqslant}
\newcommand{\id}{\mathrm{id}}

\begin{document}
\title{On the homotopy type of certain cobordism categories of surfaces}
\author{George Raptis}%
\address{Universit\"{a}t Osnabr\"{u}ck, 
Institut f\"{u}r Mathematik, 
Albrechtstrasse 28a,
49069 Osnabr\"{u}ck, Germany}
\email{graptis@mathematik.uni-osnabrueck.de}
\date{}
\maketitle

\begin{abstract}
Let $\mathcal{A}_{g,d}$ be the (topological) cobordism category of orientable surfaces whose connected components are  homeomorphic to either $S^1 \times I$ with one incoming and one outgoing boundary component or the surface $\Sigma_{g,d}$ of genus $g$ and $d$ boundary components that are all incoming. In this paper, we study the homotopy type of the classifying space of the cobordism category $\mathcal{A}_{g,d}$ and the associated (ordinary) cobordism category of its connected components $\mathbb{A}_d$. $\mathcal{A}_{0,2}$ is the cobordism category of complex annuli that was considered by Costello and $\mathbb{A}_2$ is homotopy equivalent with the positive boundary 1-dimensional embedded cobordism category of Galatius-Madsen-Tillmann-Weiss. We identify their homotopy type with the infinite loop spaces associated with certain Thom spectra. 
\end{abstract}

\section{Introduction}

The concept of a \textit{cobordism category} emerged from the works of Atiyah \cite{At} and Segal \cite{Se2}  to define topological and conformal field theories respectively.  Broadly speaking, a cobordism category is a category where the objects are given by closed manifolds and the morphisms are cobordisms between them. The cobordism categories that appear in the definition of a topological field theory are only concerned with the topological type of the relevant manifolds, so they are ordinary categories. On the other hand, the cobordism category of Riemann surfaces $\mathcal{M}$ that was introduced in \cite{Se2} is a topological category that has as objects the non-negative integers and the morphism spaces are given by the moduli spaces of Riemann surfaces. The authors of \cite{GMTW} defined  a topological cobordism category $\mathcal{C}_d$ of $d$-dimensional smooth manifolds by considering smooth cobordisms embedded in some arbitrarily high dimensional Euclidean space.  The choice of the (unparametrised) embedding induces the topology on these cobordism categories and its homotopy type is given by the homotopy type of the classifying spaces of the diffeomorphism groups of the corresponding smooth manifolds. Then their classifying spaces may be regarded as certain kinds of moduli spaces of manifolds that encode information that is stable with respect to extending by cobordisms. The homotopy type of the $d$-dimensional embedded cobordism category, along with its variations that are specified by a choice of a tangential structure, was determined in \cite{GMTW}. In the oriented $2$-dimensional case, this led to a new proof of Mumford's conjecture on the cohomology of the stable mapping class group of orientable surfaces.

This paper is concerned with the homotopy type of the classifying spaces of certain subcategories $\mathcal{A}_{g,d}$ of $\mathcal{M}$ and the associated ordinary cobordism categories of their connected components $\mathbb{A}_d$. For every $d \geq 1$ and $g \geq 0$, the cobordism category $\mathcal{A}_{g,d}$ is a model for the subcategory of $\mathcal{M}$ that has the same objects but it only contains Riemann surfaces each of whose connected components is either homeomorphic to the cylinder $S^1 \times I$ with one incoming and one outgoing boundary component or the surface $\Sigma_{g,d}$ of genus $g$ with $d$ boundary components that are all incoming. This is a natural choice of a subcategory of (the positive boundary subcategory of) $\mathcal{M}$ that isolates the contribution of a single genus $g$, while the parameter $d$ functions as a decoration. Our main result about the homotopy type of $B \mathcal{A}_{g,d}$ says that there is a homotopy fiber sequence
\begin{equation} \label{hofi2}
QB\Diff^+(\Sigma_g^{(d)})_+ \to QBSO(2)_+ \to B\mathcal{A}_{g,d}
\end{equation}
where the first map is induced by the evaluation of the differential of a diffeomorphism at $d$ 
marked points.  Here $\Diff^+(\Sigma_g^{(d)})$ is the topological group of orientation-preserving diffeomorphisms of a closed orientable surface of genus $g$ with $d$ marked points that are allowed to permute.

The cobordism category $\mathbb{A}_d$ is the category of connected components of $\mathcal{A}_{g,d}$ (this is clearly independent of $g$). It is a subcategory of the cobordism category of orientable surfaces up to homeomorphism. We show that there is a homotopy fiber sequence 
\begin{equation} \label{hofi1}
QB\Sigma_d{}_+ \to QS^0 \to B\mathbb{A}_d
\end{equation}
where the first map is induced by the natural map $B \Sigma_d \to B \Sigma_{\infty}^+ \simeq Q_dS^0$. Here $Q_d S^0$ is the connected component of $QS^0$ of the map with degree $d$. 

Intuitively, we may regard the cobordism categories $\mathcal{A}_{g,d}$ and $\mathbb{A}_d$ as admitting a presentation with generators given by the isomorphisms and relations that are generated by the space of morphisms from $d$ to $0$. Then one can interpret these homotopy fiber sequences as resolutions of the cobordism categories $\mathcal{A}_{g,d}$ and $\mathbb{A}_d$ by the free infinite loop spaces generated respectively by the generators and the relations in the presentation. 

Two cases of special interest are the topological cobordism category of complex annuli $\mathcal{A}_{0,2}$ and the ordinary cobordism category of its connected components $\mathbb{A}_2$. The category of complex annuli appeared in the work of Costello \cite{Cos2} in connection with the homotopy type of the compactified moduli spaces of Riemann surfaces. In this case, the homotopy fiber sequence \eqref{hofi2} can be written as follows,  
\begin{displaymath}
QBO(2)_+ \to QBSO(2)_+ \to B \mathcal{A}_2
\end{displaymath}
and the first map is the stable transfer map associated with the double covering (up to homotopy) $BSO(2) \to BO(2)$. On the other hand, the cobordism category $\mathbb{A}_2$ will be seen to be homotopy equivalent with the positive boundary 1-dimensional embedded cobordism category $\mathcal{C}_{1, \partial}$ of \cite{GMTW}. For $d=2$, the homotopy fiber sequence \eqref{hofi1} can be written as follows,
\begin{displaymath}
Q\R P^{\infty}_+ \to QS^0 \to B\mathbb{A}_2
\end{displaymath}
and the first map is the stable transfer map associated with the universal double covering. As a consequence, we can identify the homotopy type of these two cobordism categories with the infinite loop spaces of the Thom spectra associated with the stable bundles that are inverse to the relevant line bundles, i.e., there are weak homotopy equivalences
\begin{align*}
B \mathcal{A}_{0,2} \stackrel{\simeq}{\to} \Omega^{\infty - 1} \mathbf{Th}(- \delta_2), \\
B \mathbb{A}_2 \stackrel{\simeq}{\to} \Omega^{\infty -1} \mathbf{Th}(- \gamma_1).
\end{align*}
Here $\delta_2$ denotes the determinant line bundle associated with the double covering $BSO(2) \to BO(2)$ and $\gamma_1$ is the universal line bundle. The second homotopy equivalence together with main result of \cite{GMTW} in the $1$-dimensional case implies that the inclusion of the positive boundary subcategory
\begin{equation} \label{hoeq}
 B\mathcal{C}_{1, \partial} \stackrel{\simeq}{\to} B \mathcal{C}_{1}
\end{equation}
is weak homotopy equivalence. By a similar argument, we show that this is true also in the oriented $1$-dimensional case, i.e.. the inclusion $B\mathcal{C}^+_{1, \partial} \stackrel{\simeq}{\to} B \mathcal{C}^+_{1}$ is a weak homotopy equivalence. The analogous results have been shown in all higher dimensions in \cite{GMTW}. \\

\parindent=0in \textit{Organisation of the paper.} In section \ref{cob2}, we discuss the homotopy theory of the classifying spaces of categories and develop the general method that will be needed in showing the existence of the homotopy fiber sequences.  This part of the paper is independent of the context of cobordism categories and so the results are also of independent interest. However we have restricted to a generality that will be sufficient for the purpose of proving our main results. Further details and natural generalisations will be discussed in the appendix. In section \ref{cob3}, we study the homotopy type of the ordinary cobordism categories $\mathbb{A}_d$ and prove the homotopy fiber sequence \eqref{hofi1}.  In section \ref{cob4}, we show that the category $\mathbb{A}_2$ is homotopy equivalent with the positive boundary 1-dimensional embedded cobordism category $\mathcal{C}_{1,\partial}$ of \cite{GMTW}. Using results from \cite{GMTW}, we then prove the weak homotopy equivalence \eqref{hoeq} and its analogue in the oriented case. In section \ref{cob5}, we study the homotopy type of the topological cobordism categories $\mathcal{A}_{g,d}$ and prove the homotopy fiber sequence \eqref{hofi2}. For technical reasons, we will work with mapping class groups instead of diffeomorphism groups of surfaces whenever the values of $g$ and $d$ are so that the components of the diffeomorphism groups are contractible. \\

\parindent=0.2in

\section{Preliminaries on Classifying Spaces}  \label{cob2}

\subsection{The Grothendieck construction} \label{cob2.0} Let $\mathcal{C}at(\mathcal{T}op)$ denote the topological category of small topological categories and continuous functors. A functor between topological categories is called continuous if it defines continuous maps between the spaces of objects and morphisms. Throughout the paper, we will work with compactly generated spaces and all constructions will be assumed to take place in this category of spaces.  

Let $\mathcal{C}$ be a small topological category in $\mathcal{C}at(\mathcal{T}op)$ with a \textit{discrete} space of objects and such that the inclusion $\Ob \mathcal{C} \hookrightarrow \Mor \mathcal{C}$ is a cofibration. The Grothendieck construction associates to every continuous functor $F: \mathcal{C} \to \mathcal{C}at(\mathcal{T}op)$, a topological category $\int_{\mathcal{C}} F$ together with a continuous projection functor $\int_{\mathcal{C}} F \to \mathcal{C}$.

The objects of $\int_{\mathcal{C}} F$ are the pairs $(c, a)$ where $c \in \Ob\mathcal{C}$ and $a \in \Ob F(c)$. A morphism from $(c,a)$ to $(c',a')$ is given by a pair $(k,f)$ where $k:c \to c'$ is a morphism in $\mathcal{C}$ and $f: F(k)(a) \to a'$ is a morphism in $F(c')$. The composition of two morphisms $(k,f): (c,a) \to (c',a')$ and $(k',f'): (c',a') \to (c'',a'')$ is defined to be the morphism $(k'  k, f'  F(k')(f))$. The topology on $\int_{\mathcal{C}} F$ is defined as follows: the space of objects $\Ob \int_{\mathcal{C}} F$ is identified with $\bigsqcup_{c \in \Ob \mathcal{C}} \Ob F(c)$ and the morphisms are topologised as subspaces of the morphism spaces of $\mathcal{C}$ and $F(c)$ for all $c \in \Ob\mathcal{C}$. There is a projection functor $p_F: \int_{\mathcal{C}} F \to \mathcal{C}$ defined on objects by sending $(c,a) \mapsto c$  and on morphisms by $(k,f) \mapsto k$. Note that its fiber at $c \in \Ob\mathcal{C}$ is the category $F(c)$. Furthermore, the construction is natural in $F$, i.e., it defines a functor $\int_{\mathcal{C}} : \mathrm{Fun}(\mathcal{C}, 
\mathcal{C}at(\mathcal{T}op)) \to \mathcal{C}at(\mathcal{T}op) \downarrow \mathcal{C}$.  

The nerve functor $N : \mathcal{C}at(\mathcal{T}op)  \to s\mathcal{T}op := \mathcal{T}op^{\Delta^{op}}$ takes a small topological category $\mathcal{D}$ to the simplicial space $N_. \mathcal{D}$ whose space of $n$-simplices is the space of sequences of $n$ composable morphisms in $\mathcal{D}$, 
\begin{displaymath}
N_n \mathcal{D} := \underbrace{\Mor\mathcal{D}_{\Ob \mathcal{D}} \times_{\Ob \mathcal{D}} \cdots \times_{\Ob \mathcal{D}} \Mor\mathcal{D}}_{n}.
\end{displaymath}
The classifying space functor $B: \mathcal{C}at(\mathcal{T}op) \to \mathcal{T}op$ is the composition of the nerve functor followed by the geometric realisation functor $|-|: s\mathcal{T}op \to \mathcal{T}op$.

The Bousfield-Kan homotopy colimit of a continuous functor $F: \mathcal{C} \to s\mathcal{T}op$, denoted by $\ho_{\mathcal{C}}F$, is (the geometric realisation of) a bisimplicial space whose $(p,q)$-simplices are the pairs
\begin{displaymath}
(c_0 \to \cdots \to c_p \in N_p \mathcal{C}, x \in F(c_0)_q).
\end{displaymath} 

For every continuous functor $F: \mathcal{C} \to \mathcal{T}op$, there is a topological category $\mathcal{C} \wr F$, called the transport or wreath-product category, whose space of objects is
\begin{displaymath}
 \Ob \mathcal{C} \wr F = \bigsqcup_{c \in \Ob \mathcal{C}} F(c)
\end{displaymath}
and  
\begin{displaymath}
 \mathcal{C} \wr F((c,x),(c',x'))= \{ f \in \mathcal{C}(c,c') : F(f)(x)=x' \}.
\end{displaymath}
 
\begin{proposition} \label{thomason-grothendieck}
For every continuous functor $F: \mathcal{C} \to \mathcal{C}at(\mathcal{T}op)$, there are natural weak homotopy equivalences
\begin{displaymath}
\xymatrix{
B(\mathcal{C} \wr BF) & |\ho_{\mathcal{C}}NF| \ar[l]_-{\simeq} \ar[r]^-{\simeq} & B(\int_{\mathcal{C}} F).
}
\end{displaymath}
\end{proposition}
\begin{proof} The first weak homotopy equivalence is clear. For the second, the proof is exactly the same as for Thomason's well-known homotopy colimit theorem \cite{T1}. (A little explanation about where the technical conditions on $\mathcal{C}$ come in here. The condition that $\Ob \mathcal{C}$ is discrete is required so that, for any choice of objects $c_0, c_1, ..., c_p \in \Ob \mathcal{C}$, the space $\prod_{0 \leq i \leq p-1} \Mor_{\mathcal{C}}(c_i, c_{i+1})$ defines a connected component of $N_p\mathcal{C}$. This condition is not strictly necessary, but it will suffice for our purposes here. The cofibration condition on $\mathcal{C}$ is a standard requirement in order to deal with classifying spaces in a \textit{good} way as in \cite[Appendix A]{Se3}.).
\end{proof}

Therefore any of these constructions gives an equally good model for the homotopy colimit of a diagram of topological categories. The advantage of the Grothendieck construction is that it produces again a 
topological category rather than a simplicial or topological space. The results of this section will be 
stated in terms of the Grothendieck construction, however we will not insist on distinguishing between the models
in later sections.

A morphism $F: \mathcal{D}_1 \to \mathcal{D}_2$ in $\mathcal{C}at(\mathcal{T}op)$ is called a weak equivalence if $B(F)$ is a weak homotopy equivalence in $\mathcal{T}op$. A natural transformation $\eta: F \to G$ of continuous functors $F, G: \mathcal{C} \to \mathcal{C}at(\mathcal{T}op)$ is called a (pointwise) weak equivalence if each of its components is a weak equivalence. 

\begin{corollary} \label{grothendieck-thomason2}
The functor $\int_{\mathcal{C}} \colon  \mathrm{Fun}(\mathcal{C}, \mathcal{C}at(\mathcal{T}op)) \to \mathcal{C}at(\mathcal{T}op)$ sends (pointwise) weak equivalences to weak equivalences.
\end{corollary}

\subsection{The little comma categories} \label{cob2.1} For every pair of objects $c$ and $c'$ in $\mathcal{C}$, consider the topological category $c \downarrow_{\mathcal{C}} c'$ whose objects are the morphisms $c \rightarrow c'$ in $\mathcal{C}$, and a morphism from $c \stackrel{u}{\to} c'$ to $c \stackrel{v}{\to} c'$ is given by a morphism $c \stackrel{f}{\to} c$ over $c'$, i.e., so that the triangle
\begin{displaymath}
\xymatrix{
c \ar[r]^f \ar[d]^u  & c \ar[dl]^v \\
c' & }
\end{displaymath}
commutes in $\mathcal{C}$. The composition of morphisms in $c \downarrow_{\mathcal{C}} c'$ is defined by the composition of morphisms in $\mathcal{C}$. The category $c \downarrow_{\mathcal{C}} c'$ is a subcategory of the category of morphisms $\mathcal{C}^{\to}$ with the induced subspace topology. 

 For every object $c \in \Ob\mathcal{C}$, there is a continuous functor
\begin{displaymath}
\mathcal{F}_c: \mathcal{C} \rightarrow \mathcal{C}at(\mathcal{T}op)
\end{displaymath}
which is defined on objects by
\begin{displaymath}
c' \mapsto c \downarrow_{\mathcal{C}} c'.
\end{displaymath}
Every morphism $c' \stackrel{g}{\to} c''$ in $\mathcal{C}$ induces a continuous functor $g_{\ast}: c \downarrow_{\mathcal{C}} c' \rightarrow c \downarrow_{\mathcal{C}} c''$ by post-composing with $g$.  Let $\End_{\mathcal{C}}(c)=\mathrm{Hom}_{\mathcal{C}}(c,c)$ denote the topological monoid of endomorphisms of $c \in \Ob\mathcal{C}$. Note that the little comma category $c \downarrow_{\mathcal{C}} c'$ is the transport category associated to the right action of $\End_{\mathcal{C}}(c)$ on the space $\mathcal{C}(c,c')$.

There is a functor
\begin{displaymath}
\xymatrix{
U(c): \int_{\mathcal{C}} \mathcal{F}_c \ar[r] & \End_{\mathcal{C}}(c)
}
\end{displaymath}
that forgets the objects and takes every morphism to the associated endomorphism of $c$.

\begin{lemma}  \label{Y}
The functor $U(c): \int_{\mathcal{C}} \mathcal{F}_c \to \End_{\mathcal{C}}(c)$  is a weak equivalence, i.e., it induces a (weak) homotopy equivalence 
\begin{displaymath}
\xymatrix{
u(c): B (\int_{\mathcal{C}} \mathcal{F}_c)  \ar[r]^-{\simeq} & B \End_{\mathcal{C}}(c).
}
\end{displaymath}
\end{lemma}

We will give two proofs of Lemma \ref{Y}. 

\begin{proof}[First Proof] Let $X_{**}$ denote the Bousfield-Kan homotopy colimit of the diagram $N \mathcal{F}_c: \mathcal{C} \to s\mathcal{T}op$. A $(p,q)$-simplex is a pair
\begin{displaymath}
(c_0 \to ... \to c_p \in N_p \mathcal{C},  c \to ... \to c \to c_0 \in N_q(c \downarrow_{\mathcal{C}} c_0))
\end{displaymath}
and the face and degeneracy maps are given by composition and insertion of identities respectively. The only non-obvious face map is the following:
\begin{displaymath}
d_0 \times d_q(c_0 \stackrel{f_1}{\to} ... \to c_p, c \to ... \stackrel{\alpha_q}{\to} c \stackrel{f}{\to} c_0)=(c_1 \to ... \to c_p,  c \to ...  \stackrel{f_1 f \alpha_q}{\longrightarrow} c_1).
\end{displaymath}
By Proposition \ref{thomason-grothendieck}, there is a weak homotopy equivalence $\theta: |X_{**}| \to B( \int_{\mathcal{C}}\mathcal{F}_c)$.  We claim that there is also a weak homotopy equivalence $\psi: |X_{**}| \to B\End_{\mathcal{C}}(c)$.  Let $Y_{**}$ be a bisimplicial space whose $(p,q)$-simplices are the pairs
\begin{displaymath}
(c \to c_0 \to ... \to c_p,  c \stackrel{\alpha_1}{\to} ... \stackrel{\alpha_q}{\to} c)
\end{displaymath}
and the face and degeneracy maps are such that the maps
\begin{displaymath}
\beta_{pq} : X_{pq} \rightarrow Y_{pq}
\end{displaymath} 
\begin{displaymath}
(c_0 \stackrel{f_1}{\to} ... \to c_p, c \to ... \stackrel{\alpha_q}{\to} c \stackrel{f}{\to} c_0) \mapsto (c \stackrel{f}{\to} c_0 \stackrel{f_1}{\to} ... \to c_p,  c \stackrel{\alpha_1}{\to} ... \stackrel{\alpha_q}{\to} c)
\end{displaymath}
define an isomorphism of bisimplicial spaces. Let $N \End_{\mathcal{C}}(c)$ be the the nerve of $\End_{\mathcal{C}}(c)$ considered as a bisimplicial space that is constant in one direction and consider the projection on the $q$-direction $\psi: Y_{**} \to N \End_{\mathcal{C}}(c)$. If we realise with respect to the $p$-direction, we obtain a map of 
simplicial spaces that is given levelwise by the projections
\begin{displaymath}
B(c \downarrow \mathcal{C}) \times N_q \End_{\mathcal{C}}(c) \to N_q \End_{\mathcal{C}}(c).
\end{displaymath} 
Since the category $c \downarrow \mathcal{C}$ has an initial object, its classifying space is contractible and so the projections are homotopy equivalences. Therefore the bisimplicial map $\psi$ is a weak homotopy equivalence. It can be checked easily that the map $u(c) \theta$ is homotopic to $\psi$. Hence it follows that $u(c)$ is a weak homotopy equivalence.
\end{proof}

\begin{proof}[Second Proof] There is a functor $G: \End_{\mathcal{C}}(c) \to \int_{\mathcal{C}} \mathcal{F}_c$ defined as follows: the unique object of $\End_{\mathcal{C}}(c)$ is sent to $(c, 1_c)$, and an endomorphism $\sigma: c \to c$ is sent to $(\sigma, \sigma)$, i.e., the morphism
\begin{displaymath}
\xymatrix{
c \ar[d]^{1} \ar[r]^{\sigma} & c \ar[d]^{1} \\
c \ar[r]^{\sigma} & c 
}
\end{displaymath}
Then we have $U G = 1$, and therefore $(BU)(BG) = 1$. Furthermore, the pair $(G,U)$ defines an continuous adjunction. In particular, there is a natural transformation $\rho: G U \rightarrow 1$ whose components are defined as follows: given $(c', \alpha) \in \Ob \int_{\mathcal{C}} \mathcal{F}_c$, $\rho_c$ is $(\alpha, 1_c)$, i.e., the morphism represented by the diagram
\begin{displaymath}
\xymatrix{
c \ar[r]^{1} \ar[d]^{1}  & c \ar[d]^{\alpha} \\
c \ar[r]^{\alpha} & c' 
}
\end{displaymath}
It is well-known that a natural transformation of functors induces a homotopy between the maps of classifying spaces (e.g. see \cite{Se}). Therefore $(BG) (BU) \simeq 1$, and the result follows.
\end{proof}

\begin{remark} Lemma \ref{Y} may be compared with the classical Yoneda lemma of category theory  (see also Proposition \ref{superyoneda}). The Yoneda lemma says that, for any (ordinary, locally small) category $C$, the canonical functor $Y: C^{op} \rightarrow Fun(C, \mathcal{S}et)$, defined on objects by $c \mapsto (c' \mapsto $Hom$_{C}(c,c'))$, is fully faithful. As a consequence, if two representable functors are naturally isomorphic then they are represented by isomorphic objects. The functor $\mathcal{F}_c$ may be regarded as a topological extension of the representable functor represented by $c \in \Ob C$. Lemma \ref{Y} implies that two such functors are weakly equivalent only if the elements that represent them have weakly equivalent endomorphism monoids. 
\end{remark}

\begin{remark} When $\mathcal{C}$ is an ordinary small category, Lemma \ref{Y} can be seen as an application of Thomason's theorem \cite{T1} and the following central idea from Quillen's Theorem B \cite{Q}. This says that, given a morphism $F: C \to D$ in $\mathcal{C}at$, the canonical map $\ho_{d \in D} B( C \downarrow d) \to BC$ is a weak homotopy equivalence. The last lemma is an application of this statement to the inclusion $\End_{\mathcal{C}}(c) \to \mathcal{C}$. Note that the second proof actually shows that the map $u(c)$ is always a homotopy equivalence, 
and not just a weak one.
\end{remark}

\subsection{Stability} \label{cob2.2} While a morphism $\alpha: c' \rightarrow c$ always induces a function $\Ob(c \downarrow_{\mathcal{C}} d) \rightarrow \Ob(c' \downarrow_{\mathcal{C}} d)$ which is natural in $d$, it is not true that this function extends canonically to a functor. In other words, there is no canonical way to choose a natural transformation $\mathcal{F}_c \rightarrow \mathcal{F}_{c'}$ for every morphism $c' \rightarrow c$. 

A natural transformation $\eta: \mathcal{F}_c  \rightarrow \mathcal{F}_{c'}$ is uniquely specified by the following data (see Proposition \ref{superyoneda}): a morphism $\alpha: c' \rightarrow c$ and a continuous homomorphism $h : \End_{\mathcal{C}}(c) \rightarrow \End_{\mathcal{C}}(c')$ between topological monoids with the compatibility condition that 
\begin{displaymath}
\xymatrix{
c' \ar[r]^{h(\sigma)} \ar[d]^{\alpha} & c' \ar[d]^{\alpha} \\
c \ar[r]^{\sigma} & c
}
\end{displaymath}
commutes for all $\sigma \in \End_{\mathcal{C}}(c)$. The morphism $\alpha$ comes from the Yoneda lemma and the homomorphism $h$ is determined by $\eta$ as follows: 
\begin{displaymath}
\eta_c(\sigma \stackrel{\sigma}{\to} 1_c) = (\sigma \circ \alpha) \stackrel{h(\sigma)}{\longrightarrow} (1_c \circ \alpha). 
\end{displaymath}

\begin{definition} Say that $\mathcal{C}$ stabilises along a subcategory $\mathcal{J}$ if for every $j \stackrel{\alpha}{\to} j'$ in $\mathcal{J}$, there is a continuous assignment of a continuous homomorphism $\alpha^*: \End_{\mathcal{C}}(j') \rightarrow \End_{\mathcal{C}}(j)$ such that 
\begin{itemize}
\item[(i)] the square
\begin{displaymath}
\xymatrix{
j \ar[r]^{ \alpha^*(\sigma)} \ar[d]^{\alpha} & j \ar[d]^{\alpha} \\
j' \ar[r]^{\sigma} & j'
}
\end{displaymath}
commutes for all $\sigma \in \End_{\mathcal{C}}(j')$,
\item[(ii)] $(\beta \alpha)^* = \alpha^* \beta^*$ and $1^* = 1$.
\end{itemize}
\end{definition}

\begin{example} \label{example2.1} Let $\mathcal{C}$ be a (strict) symmetric monoidal category with unit object $1$. Every morphism $f: c \to 1$ in $\mathcal{C}$ generates a subcategory $\mathcal{I}(f)$ of $\mathcal{C}$ defined to be the smallest subcategory that contains the morphisms $1^{\otimes (n -1)} \otimes f: c^{\otimes n} \to c^{\otimes n-1}$. Let $\End_n(c)$ denote the monoid of endomorphisms of the object $c^{\otimes n}$. For every $n \geq 1$, there is a  homomorphism $\End_n(c) \to \End_{n+1}(c)$ given by tensoring with the identity $1_c: c \to c$.  Then it is easy to see that $\mathcal{C}$ stabilises along $\mathcal{I}(f)$ for every such $f$. This is the main example and it covers all the cases that we will be interested in.
\end{example}

Suppose that $\mathcal{C}$ stabilises along a subcategory $\mathcal{J}$. Then there is a functor $\mathcal{J}^{op} \times \mathcal{C} \rightarrow \mathcal{C}at(\mathcal{T}op)$ that is defined on objects by $(j,c) \mapsto j \downarrow_{\mathcal{C}} c$.  Applying the Grothendieck construction with respect to $\mathcal{J}^{op}$, we obtain a functor 
\begin{displaymath}
\mathcal{G}_{\mathcal{J}}: \mathcal{C} \rightarrow \mathcal{C}at(\mathcal{T}op)
\end{displaymath}
which is defined on objects by
\begin{displaymath}
\xymatrix{
c \mapsto \int_{\mathcal{J}^{op}} ( - \downarrow_{\mathcal{C}} c)=  \int_{\mathcal{J}^{op}} \mathcal{F}_{-} (c). 
}
\end{displaymath}

\begin{definition} Suppose that $\mathcal{C}$ stabilises along a subcategory $\mathcal{J}$. Say that $\mathcal{C}$ is $H \Z$-stable along $\mathcal{J}$ if for every morphism $u: c \to c'$ in $\mathcal{C}$, the map $B\mathcal{G}_{\mathcal{J}}(u)$ is an integral homology equivalence. 
\end{definition}

Similarly we can define $\mathcal{W}$-stability with respect to any class of weak equivalences $\mathcal{W}$ in $\mathcal{T}op$. The corresponding theory for more general classes 
$\mathcal{W}$ will be sketched in the appendix.

\begin{definition} Let $F \stackrel{i}{\to} E \stackrel{p}{\to} B$ be a sequence of maps together 
with a homotopy from the composite $p i$ to the constant map at $b \in B$. Say that it is a homology fiber sequence 
if the canonical map $F \to \hof_b(p)$ is an integral homology equivalence. 
\end{definition}

In the following results of the section, we also \textit{assume} that the space $\Mor \mathcal{C}$ is locally contractible and paracompact. It is always possible to replace $\mathcal{C}$ by a topological category that has these properties (see Remark \ref{technical-assumption}), so this assumption is only a matter of convenience here.

\begin{theorem} \label{GGCT}
If $F: \mathcal{C} \to \mathcal{C}at(\mathcal{T}op)$ is a continuous functor such that $BF(f)$ is a homology equivalence for every $f \in \Mor\mathcal{C}$,  then the sequence
\begin{displaymath}
\xymatrix{
(B F)(c) \ar[r] &  B(\int_{\mathcal{C}} F) \ar[r] & B\mathcal{C}
}
\end{displaymath}
is a homology fiber sequence. 
\end{theorem}
\begin{proof}
This is a version of the ``generalised group completion theorem'' of \cite[Theorem 3.2]{Ti} (see also \cite[Proposition 7.1]{GMTW}). A proof of a more general statement will be given in Theorem \ref{W-fiber-seq1} of the appendix. 
\end{proof}

The next theorem is the main result that we will need in the following sections.  

\begin{theorem} \label{ggct}
Suppose that $\mathcal{C}$ is $H\Z$-stable along a subcategory $\mathcal{J}$. Then there is a homology fiber sequence 
\begin{displaymath}
\xymatrix{
B\mathcal{G}_{\mathcal{J}}(c)  \ar[r] & B( \int_{\mathcal{J}^{op}} \End_{\mathcal{C}}(-))  \ar[r] & B\mathcal{C} \\
}
\end{displaymath}
and $\ho_{\mathcal{J}^{op}} B \End_{\mathcal{C}}(-) \stackrel{\simeq}{\to} B( \int_{\mathcal{J}^{op}} \End_{\mathcal{C}}(-))$.
\end{theorem}
\begin{proof}  According to Theorem \ref{GGCT}, there is a homology fiber sequence
\begin{displaymath}
\xymatrix{
(B \mathcal{G}_{\mathcal{J}})(c) \ar[r] & B(\int_{\mathcal{C}} \mathcal{G}_{\mathcal{J}}) \ar[r] & B\mathcal{C}. \\
}
\end{displaymath}
The Grothendieck construction satisfies the following commutativity property 
\begin{displaymath}
\xymatrix{
\int_{\mathcal{C}} \mathcal{G}_{\mathcal{J}} = \int_{\mathcal{C}} \int_{\mathcal{J}^{op}} \mathcal{F}_{-} (-) = \int_{\mathcal{J}^{op}} \int_{\mathcal{C}} \mathcal{F}_{-} (-) \\
}
\end{displaymath}
and so by Lemma \ref{Y} and Corollary \ref{grothendieck-thomason2} there is a weak homotopy equivalence
\begin{displaymath}
\xymatrix{
B( \int_{\mathcal{J}^{op}} \int_{\mathcal{C}} \mathcal{F}_{-} (-)) \ar[r]^-{\simeq} & B(\int_{\mathcal{J}^{op}} \End_{\mathcal{C}}(-)) .\\
}
\end{displaymath}
Hence there is a homology fiber sequence
\begin{displaymath}
\xymatrix{
B\mathcal{G}_{\mathcal{J}}(c) \ar[r] & B(\int_{\mathcal{J}^{op}} \End_{\mathcal{C}}(-)) \ar[r] & B\mathcal{C}
}
\end{displaymath}
as required, and it is induced by the functors
\begin{displaymath}
j \downarrow_{\mathcal{C}} c \rightarrow \End_{\mathcal{C}}(j) \rightarrow \mathcal{C}
\end{displaymath}
as can be seen from the proof of Lemma \ref{Y}. The last statement is a consequence of the homotopical property of the Grothendieck construction as discussed in Proposition \ref{thomason-grothendieck}.
\end{proof}

\begin{example} \label{example2.2} Consider the situation of Example \ref{example2.1}. Let $B\End_{\infty}(c) = \ho_n B\End_n(c)$. The monoid $\End_n(c)$ acts naturally on the space Hom$_{\mathcal{C}}(c^{\otimes n},1)$. Let $\mathcal{F}_n$ denote the homotopy quotient of the action. There are canonical maps $\mathcal{F}_n \to \mathcal{F}_{n+1}$, and so let $\mathcal{F}_{\infty} = \ho_n \mathcal{F}_n$. The last theorem gives sufficient conditions for the sequence $\mathcal{F}_{\infty} \to B\End_{\infty}(c) \to B\mathcal{C}$ to be a homology fiber sequence. 
\end{example}

\begin{remark}
The last theorem generalises in two ways (see Theorem \ref{W-fiber-seq2}): first, the collection of inclusions $\End_{\mathcal{C}}(j) \to \mathcal{C}$ can be replaced by a collection of functors $\mathcal{D}_j \to \mathcal{C}$ that are suitably connected together into a diagram of categories over $\mathcal{C}$. Second, the class of homology equivalences can be replaced by any class of weak equivalences that satisfies certain local-to-global properties. 
\end{remark}

\subsection{The plus construction} \label{cob2.3}
We briefly recall some facts about Quillen's plus construction.  For a detailed account of the construction and its properties, see for example \cite{Be}. Throughout this subsection, we work with non-degenerately based connected spaces $X$ that have the homotopy type of a CW-complex. 

For every perfect normal subgroup $P$ of $\pi_1(X)$, there is a space $X_P$ and an acyclic cofibration $q_P: X \rightarrow X_P$ with $\mathrm{ker}\pi_1 (q_P) = P$. The space $X_P$ is uniquely determined up to homotopy equivalence. The plus construction of $X$, denoted by $X^+$, is the space associated with the maximal perfect subgroup of $\pi_1(X)$, denoted by $P\pi_1(X)$. The plus construction is a functorial assignment: given a map $f: X \rightarrow Y$, there is a unique homotopy class of maps $f^+: X^+ \rightarrow Y^+$ that makes the following square
\begin{displaymath}
\xymatrix{
X \ar[r]^f \ar[d] & Y \ar[d] \\
X^+ \ar[r]^{f^+} & Y^+ 
}
\end{displaymath} 
commute up to homotopy. If $f: X \rightarrow Y$ is an acyclic map and the maximal perfect subgroup of $\pi_1(X)$ is trivial, then $f$ is a homotopy equivalence. In particular, $f^+$ is a homotopy equivalence for every acyclic map $f$.

\begin{proposition}  \label{L1}
If $F \rightarrow  E \rightarrow  B $ is a homotopy fiber sequence of connected spaces and $P\pi_1(B)=1$, then $F^+ \rightarrow  E^+ \rightarrow B$ is also a homotopy fiber sequence.
\end{proposition}
\begin{proof} 
See \cite[Theorem 6.4]{Be}.
\end{proof}

\begin{proposition} \label{+}
Let $F \to E \to B$ be a homology fiber sequence of connected spaces. If 
\begin{itemize}
\item[(a)] $P\pi_1(B)=1$ and
\item[(b)] $E^+$ and $F^+$ are nilpotent spaces,
\end{itemize}
then $F^+ \to E^+ \to B$ is a homotopy fiber sequence.
\end{proposition}
\begin{proof}
Let $F_f$ denote the homotopy fiber of $f$ and $g: F \to F_f$ be the induced homology equivalence. Note that $F_f$ is connected. By Proposition \ref{L1}, $F_f^+ \to E^+ \to B$ is a homotopy fiber sequence. Since $E^+$ is nilpotent by condition (b), the homotopy fiber $F_f^+$ is again nilpotent (see \cite{Dr}). Hence the map $g^+$ is a homology equivalence between nilpotent spaces, and therefore also a (weak) homotopy equivalence by the generalised Whitehead theorem \cite{Dr}. It follows that $F^+ \to E^+ \to B^+=B$ is a homotopy fiber sequence. 
\end{proof}

\section{The Cobordism Categories $\mathbb{A}_d$} \label{cob3}

\subsection{Definition of $\mathbb{A}_d$} \label{cob3.0} Let $d \geq 1$ be an integer. The category $\mathbb{A}_d$ is a subcategory of the cobordism category of orientable surfaces up to homeomorphism. The objects are the non-negative integers regarded as disjoint unions of labelled circles up to homeomorphism; a morphism from $m$ to $k$ is an  orientable $2$-dimensional cobordism up to homeomorphism with $m$ incoming and $k$ outgoing boundary circles and each of whose connected components is
\begin{itemize}
\item[(a)] either the cobordism of genus $0$ from $1$ to $1$, i.e., the cylinder $S^1 \times I$,
\item[(b)] or the surface of genus $0$ with $d$ boundary circles that are all incoming. This morphism will be denoted by $S(d): d \to 0$.
\end{itemize} 
The disjoint union of cobordisms defines a symmetric monoidal pairing on $\mathbb{A}_d$. In particular, it makes the set of connected components $\pi_0 B \mathbb{A}_d$ into an abelian monoid. $\mathbb{A}_d$ has $d$ connected components and the reduction map $\Ob \mathbb{A}_d \to \Z_d$ clearly induces an isomorphism 
\begin{displaymath}
\pi_0 B\mathbb{A}_d  \cong \Z_d.
\end{displaymath}
By well-known results in the theory of infinite loop spaces, it follows that $B\mathbb{A}_d$ is an infinite loop space (e.g. see \cite{Se3}, \cite{May2}). 

\textit{Notation.} We will write $\overline{\mathbb{A}}_d$ for the component of the category $\mathbb{A}_d$ that contains the object $0$, and $\mathbf{n}$ for the object $n \cdot d$ of $\overline{\mathbb{A}}_d$. \

\subsection{The homotopy fiber sequences} \label{cob3.1} The monoid $\End(\mathbf{k})$ is canonically isomorphic with the group of permutations of $k \cdot d$ elements. For every $\mathbf{k} \in \Ob\overline{\mathbb{A}}_d$, let $i_k: \End(\mathbf{k}) \rightarrow \End(\mathbf{k+1})$ be the homomorphism that corresponds to the canonical inclusion $\Sigma_{kd} \rightarrow \Sigma_{(k+1)d}$. Also, let $\alpha_k: \mathbf{k+1} \rightarrow \mathbf{k}$ denote the morphism $1_{\mathbf{k}} \sqcup S(d)$ where $S(d)$ is the unique morphism $\mathbf{1} \rightarrow \mathbf{0}$ and $\sqcup$ is the symmetric monoidal pairing given by the disjoint union of cobordisms. Let $\mathcal{I}$ denote the subcategory generated by the morphisms $\alpha_k$, $k \geq 0$. Note that $\mathcal{I}$ is precisely $\mathcal{I}(S(d))$ in the notation of Example \ref{example2.1}. For every $\sigma \in \End(\mathbf{k})$, we have $\sigma  \alpha_k = \alpha_k  i_k(\sigma)$ , so $\overline{\mathbb{A}}_d$ stabilises along the subcategory $\mathcal{I}$. 

The little comma category $\mathbf{n} \downarrow_{\overline{\mathbb{A}}_d} \mathbf{k}$ is a connected groupoid, and the automorphism group of each of its objects is isomorphic with the group $\Sigma_{n-k} \wr \Sigma_d$. Thus the inclusion  $\Sigma_{n-k} \wr \Sigma_d \rightarrow \mathbf{n} \downarrow_{\overline{\mathbb{A}}_d} \mathbf{k}$ of the automorphisms of the object $\mathbf{n} \rightarrow \mathbf{k}$ that is in $\mathcal{I}$ induces a homotopy equivalence between the classifying spaces. Moreover,  there is a commutative square
\begin{displaymath}
\xymatrix{
B(\mathbf{n} \downarrow_{\overline{\mathbb{A}}_d} \mathbf{k}) \ar[r]^-{B(\alpha_n^*)} & B(\mathbf{n+1} \downarrow_{\overline{\mathbb{A}}_d} \mathbf{k})  \\
B(\Sigma_{n-k} \wr \Sigma_d) \ar[u]^{\simeq} \ar[r] & B(\Sigma_{n-k+1} \wr \Sigma_d \ar[u]^{\simeq})
}
\end{displaymath}
where the top map is induced by $\alpha_n$ and $i_n$ while the bottom map is induced by the natural inclusion of 
groups. Thus the space hocolim$_{\mathcal{I}^{op}} B(- \downarrow_{\overline{\mathbb{A}}_d} \mathbf{k})$ can be identified up to homotopy with $B(\Sigma_{\infty} \wr \Sigma_d)$ for every $\mathbf{k} \in \Ob \overline{\mathbb{A}}_d$.

\begin{proposition} \label{stab}
$\overline{\mathbb{A}}_d$ is $H\Z$-stable along $\mathcal{I}$, i.e., for every morphism $u: \mathbf{k} \to \mathbf{k'}$ the canonical map 
$$\h B(-\downarrow_{\overline{\mathbb{A}}_d} \mathbf{k}) \to 
\h B(- \downarrow_{\overline{\mathbb{A}}_d} \mathbf{k'})$$
is a homology equivalence. 
\end{proposition}
\begin{proof}
It suffices to prove this for a morphism $u: \mathbf{k+1} \to \mathbf{k}$. There is a commutative square
\begin{displaymath}
\xymatrix{
B(\mathbf{n} \downarrow_{\overline{\mathbb{A}}_d} \mathbf{k+1}) \ar[d] \ar[r] & B(\mathbf{n+1} \downarrow_{\overline{\mathbb{A}}_d} \mathbf{k+1}) \ar[d]  \\
B(\mathbf{n} \downarrow_{\overline{\mathbb{A}}_d} \mathbf{k}) \ar[r] & B(\mathbf{n+1} \downarrow_{\overline{\mathbb{A}}_d} \mathbf{k})  
}
\end{displaymath} 
where the vertical maps are induced by $u$. This can be identified up to homotopy equivalence with a square  
\begin{displaymath}
\xymatrix{
B(\Sigma_{n-(k+1)} \wr \Sigma_d) \ar[d] \ar[r] & B(\Sigma_{n+1-(k+1)} \wr \Sigma_d) \ar[d]  \\
B(\Sigma_{n-k} \wr \Sigma_d) \ar[r] & B(\Sigma_{n+1-k} \wr \Sigma_d)
}
\end{displaymath}
where the horizontal maps are the canonical inclusions, but the vertical maps are conjugates of the canonical inclusions, i.e., they are given by the canonical inclusions of groups followed by an inner automorphism that re-labels the elements. Notice that an inner automorphism cannot be chosen uniformly for all $n$. But since conjugation by an element does not affect the map on homology, it follows that the induced map between the colimits is a homology equivalence, and therefore so is the map $\h B(-\downarrow_{\overline{\mathbb{A}}_d} \mathbf{k+1}) \to 
\h B(- \downarrow_{\overline{\mathbb{A}}_d} \mathbf{k})$.
\end{proof}

Let $A_n \unlhd \Sigma_n$ denote the alternating group. We will write $A_n \wr^e \Sigma_d$ for the subgroup of $\Sigma_n \wr \Sigma_d$ that consists of the elements $(\sigma; \beta_1, ..., \beta_n)$ such that $\sigma \in A_n$ and the number of odd permutations among the $\beta_k$'s is even. 

\begin{proposition} \label{perf}
$(i)$ Let $n \geq 5$ and $d \geq 2$. The maximal perfect subgroup of $ {\Sigma}_n \wr \Sigma_d$ is $A_n \wr^e \Sigma_d$. The quotient is isomorphic with $\Z_2 \times \Z_2$. \
$(ii)$ The maximal perfect subgroup of $ {\Sigma}_{\infty}\wr \Sigma_d$ is
$A_{\infty} \wr^e \Sigma_d$. The quotient is isomorphic with $\Z_2 \times \Z_2$.
\end{proposition}
\begin{proof} $(i)$ The maximal perfect subgroup $P$ of $\Sigma_n \wr \Sigma_d$ is a subgroup of $A_n \wr^e \Sigma_d$ because every commutator in $\Sigma_n \wr \Sigma_d$ lies in $A_n \wr ^e \Sigma_d$. Therefore it suffices to show that $A_n \wr^e \Sigma_d$ is perfect. It is well-known that $A_n$ is a perfect group for all $n \geq 5$. Hence the commutator subgroup of $A_n \wr^e \Sigma_d$ contains the subgroup $A_n \wr A_d$. Then it suffices to note that every element of the form
\begin{displaymath}
(\id;(ab), (ab), \id, \cdots, \id)
\end{displaymath}
can be written as a commutator of the elements 
\begin{align*}
((12)(34) ; \id, \cdots, \id), && (\id;(ab),\id, \cdots, (ab)).
\end{align*}
The second claim is now obvious: there is a surjective homomorphism $\phi: \Sigma_n \wr \Sigma_d \to \Z_2 \times \Z_2$ defined by $\phi( \sigma ; \beta _1,..., \beta_n) = ( sign( \sigma ), sign( \beta_1 \cdots  \beta_n) )$ whose kernel is $A_n \wr^e \Sigma_d$. $(ii)$ follows immediately from $(i)$.
\end{proof}

Let $QX_+:= \ho_n\Omega^n \Sigma^n X_+$ denote the free infinite loop space generated by the (unbased) space $X$ and $Q_n X_+$ be the connected component of the map of degree $n$ when $X$ is connected. From the theory of infinite loop spaces (see \cite{Ad},\cite{May}, \cite{Se4}), this infinite loop space is weakly homotopy equivalent with the group completion of a suitably defined topological monoid
\begin{displaymath}
C_{\infty}(X_+) \simeq \bigsqcup_n E\Sigma_n \times_{\Sigma_n} X^n
\end{displaymath}
of configurations of points in $\R^{\infty}$ labelled by points in $X$. In the case where $X=B\Sigma_d$, there is a canoncal  identification $B(\Sigma_n \wr \Sigma_d) \cong E\Sigma_n \times_{\Sigma_n} (B\Sigma_d)^n$, and so by the group completion theorem \cite{McDS}, there is a weak homotopy equivalence
\begin{displaymath}
\Z \times B(\Sigma_{\infty} \wr \Sigma_d)^+ \stackrel{\simeq}{\to} Q(B\Sigma_d){}_+.
\end{displaymath}
In particular, the spaces $B(\Sigma_{\infty} \wr \Sigma_d )^+$ are clearly nilpotent for all $d$. \

\begin{proposition} \label{main1}
There is a homotopy fiber sequence
$B(\Sigma_{\infty} \wr \Sigma_d){}^+ \stackrel{\overline{\tau}_d}{\to} B \Sigma_{\infty}^+ \to B \overline{\mathbb{A}}_d$ 
where the map $\overline{\tau}_d$ is induced by the natural homomorphisms $\Sigma_n \wr \Sigma_d \to \Sigma_{nd}$.
\end{proposition}
\begin{proof}
By Theorem \ref{ggct} and Proposition \ref{stab}, there is a homology fiber sequence
$$B({\Sigma}_{\infty}\wr \Sigma_d) \rightarrow B{\Sigma}_{\infty} \rightarrow B\overline{\mathbb{A}}_d.$$ 
The conditions of Proposition \ref{+} are satisfied and therefore
\begin{displaymath}
B(\Sigma_{\infty} \wr \Sigma_d) ^+ \to B\Sigma_{\infty} ^+ \to B\overline{\mathbb{A}}_d
\end{displaymath}
is a homotopy fiber sequence, as required. According to Theorem \ref{ggct}, the first map is induced by the collection of forgetful functors $\mathbf{n} \downarrow_{\overline{\mathbb{A}}_d} \mathbf{0} \to \End(\mathbf{n})$ that can be identified (up to an equivalence of categories) with the homomorphisms $\Sigma_n \wr \Sigma_d \to \Sigma_{nd}$.
\end{proof} 

By Proposition \ref{perf}, we can identify the homotopy fiber sequence at the level of the fundamental groups as follows:
\begin{itemize}
\item if $d$ is even, then
$\pi_1(B({\Sigma}_{\infty}\wr \Sigma_d)^+) \cong \Z_2 \times \Z_2  \rightarrow \pi_1({B{\Sigma}_{\infty}}^+) \cong \Z_2$ is the projection onto a factor because an element $(\sigma; \beta_1, \cdots, \beta_n)$ in $\Sigma_{\infty} \wr \Sigma_d$ gives rise to an odd permutation in $\Sigma_{\infty}$ if and only if $\{\beta_k\}_{1 \leq k \leq n}$ contains an odd number of odd permutations,
\item if $d$ is odd, it is given by the addition map in $\Z_2$. 
\end{itemize}
By the long exact sequence of homotopy groups, it follows that 
\begin{displaymath}
\pi_1(B\overline{\mathbb{A}}_d)=\{0\}
\end{displaymath}
for all $d \geq 1$.

Since the spaces $B(\Sigma_{\infty} \wr \Sigma_d)^+$ and $B\Sigma_{\infty} ^+$ can be identified up to weak homotopy equivalence with a connected component of the free infinite loop spaces $QB\Sigma_d{}_+$ and $QS^0$ respectively, so the map $$[d] \times \overline{\tau}_d: \Z \times B({\Sigma}_{\infty} \wr \Sigma_d)^+  \rightarrow \Z \times B{\Sigma}_{\infty}^+$$ can be identified with the infinite loop map $$\tau_d: Q(B\Sigma_d)_+ \to QS^0$$ that is induced by the canonical map $B\Sigma_d \to QS^0$. Thus we conclude with the following theorem which is essentially a reformulation of Proposition \ref{main1}.

\begin{theorem} \label{mainA}
There is a homotopy fiber sequence of infinite loop spaces 
\begin{displaymath}
Q(B\Sigma_d){}_+ \stackrel{\tau_d}{\to} QS^0 \to B \mathbb{A}_d.
\end{displaymath}
\end{theorem}

\begin{remark} ($d=1$) Note that the category $\mathbb{A}_1$ is contractible because $0$ is a terminal object.
\end{remark}

\subsection{The stable transfer map} \label{cob3.2}

 Let $p:X \rightarrow B$ be a covering map of CW-complexes with finite fibers of size $n$. There is a construction of a stable map $tr^s: QB_+ \rightarrow QX_+$, called the stable transfer map \cite{Ad}, \cite{KP}, \cite{Ro}.

We recall the construction of the stable transfer map. Firstly, define the pretransfer map $ptr : B \rightarrow E\Sigma_n \times_{\Sigma_n} X^n$ as follows: let $\bar{X} \rightarrow B$ be the $\Sigma_n$-bundle associated with the covering map $p$, i.e., $\bar{X}=\{(x_1,...,x_n) \in X^n : \{x_1, \cdots, x_n\}=p^{-1}(p(x_1)) \}$, and let $f:B \rightarrow B\Sigma_n$ be the map that classifies this $\Sigma_n$-bundle. There are $\Sigma_n$-maps $\bar{X} \rightarrow X^n$ and $\bar{X} \rightarrow E\Sigma_n$. Quotienting out the $\Sigma_n$-action on both sides of the map $\bar{X} \rightarrow E\Sigma_n \times X^n$, we obtain the pretranfer map $ptr$. 

For any generalised cohomology theory $h^*$, there is a transfer homomorphism $p_!:h^*(X_+) \rightarrow h^*(B_+)$ in the ``reverse'' direction. It is defined as follows: let $\alpha \in h^m(X_+)$ be a cohomology class represented by a map $\bar{\alpha} : X \to K_m$, where $\{K_m\}$ is the $\Omega$-spectrum that represents $h^*$. Then $p_!(\alpha)$ is defined to be the cohomology class represented by
\begin{displaymath}
 B \stackrel{ptr}{\rightarrow} E\Sigma_n \times_{\Sigma_n} X^n \rightarrow E\Sigma_n \times_{\Sigma_n} K_m^n \rightarrow K_m
\end{displaymath}
where the last map is defined by the $n$-fold loop structure of $K_m \cong \Omega^n K_{m+n}$.

Let $h^*$ be the cohomology theory represented by the infinite loop space $QX_+$. The transfer map is defined to be 
\begin{displaymath}
tr:=p_!(i_X)
\end{displaymath}
where $i_X: X \rightarrow QX_+$ denotes the canonical inclusion.

The stable transfer map $tr^s: QB_+ \to QX_+$ is the unique map of infinite loop spaces that is induced by the transfer map. In detail, it is defined by the collection of maps
\begin{displaymath} 
E\Sigma_m \times_{\Sigma_m} B^m \stackrel{1 \times (ptr)^m}{\longrightarrow} E\Sigma_m \times_{\Sigma_m} (E\Sigma_n \times_{\Sigma_n} X^n)^m \rightarrow E\Sigma_{mn} \times_{\Sigma_{mn}} X^{mn}
\end{displaymath}
for all $m \geq 0$. The last map comes from the operadic action of $\{E\Sigma_n\}_n$ on $QX_+$ or, in other words, the infinite loop space structure of $QX_+$. In order to interpret these maps geometrically, it is useful to think of the space of ordered configurations of $n$ points in $\R^{\infty}$ as a concrete model for $E\Sigma_n$. 

\begin{proposition} \label{tr} The map $\tau_{d+1}: Q(B\Sigma_{d+1})_+ \to QS^0$ is the composition of the stable transfer map $Q(B\Sigma_{d+1}){}_+ \to Q(B\Sigma_{d}){}_+$ of the covering up to homotopy $B\Sigma_d \to B\Sigma_{d+1}$ followed by the projection map $Q(B\Sigma_d){}_+ \to QS^0$ that is induced by $B\Sigma_d \to pt$. In particular,  $\tau_2: Q(B\Sigma_2){}_+ \to QS^0$ is the stable transfer map of the universal double covering.
\end{proposition}
\begin{proof}
The composition of the two stable maps is induced by the composition of the pretransfer map $B\Sigma_{d+1} \to E\Sigma_{d+1} \times_{\Sigma_{d+1}} (B\Sigma_d)^{d+1}$ followed by the projection $E\Sigma_{d+1} \times_{\Sigma_{d+1}} (B\Sigma_d)^{d+1} \to E\Sigma_{d+1} \times_{\Sigma_{d+1}} * = B\Sigma_{d+1} \to QS^0$. This is also the map that induces $\tau_{d+1}$, hence result.
\end{proof}

A geometric interpretation of the map $\tau_d: Q(B\Sigma_d)_+ \rightarrow QS^0$ can be given in terms of framed bordism. The framed bordism groups $\Omega^{fr}_*(X)$ of a (unbased) space $X$ define a (unreduced) generalised homology theory that is represented by the sphere spectrum, i.e. $\Omega^{fr}_n(X) \cong \pi_n(QX_+)$. The elements of $\Omega^{fr}_n(X)$ are represented by triples $(M, \phi;f)$ where $M$ is a closed smooth $n$-dimensional manifold, $\phi$ is a trivialisation of its stable normal bundle (i.e., a lifting of the stable normal bundle $\bar{\nu} : M \rightarrow BO(N)$ to $EO(N)$) and a map $f: M \rightarrow X$, up to framed bordisms in $X$. The homomorphism $\pi_n(\tau_d) \colon \pi_n Q(B \Sigma_d)_+ \rightarrow \pi_nQS^0$ takes a class $(M,\phi;f)$ to the class $(M',\phi')$ where $M' \rightarrow M$ is the $d$-fold covering pulled back from $B \Sigma_{d-1} \to B \Sigma_d$ along the map $f$, and $\phi'$ is the trivialisation of the stable normal bundle of $M'$ that is induced by $\phi$. 

\begin{corollary} \label{htpygps}
There are short exact sequences of abelian groups
\begin{displaymath}
0 \rightarrow \pi_{n+1} B\mathbb{A}_d \rightarrow \Omega^{fr}_n(B\Sigma_d) \stackrel{\pi_n(\tau_d)}{\longrightarrow} \Omega^{fr}_n(pt) \rightarrow 0
\end{displaymath}
for every $n \geq 1$.
\end{corollary}
\begin{proof}
By the Kahn-Priddy theorem (see \cite{KP}, \cite{Se5}), the homomorphism $\Omega^{fr}_n(B \Sigma_d) \rightarrow \Omega^{fr}_n(pt)$ is surjective for all $n \geq 1$. (Note that it suffices to know this for $d=2$. For $n=0$, the map is given by multiplication by $d$.). Then the result follows from the long exact sequence of homotopy groups associated with the homotopy fiber sequence of Theorem \ref{mainA}.
\end{proof}

\begin{corollary}$\pi_2(B\mathbb{A}_d)\cong \Z_2$ for every $d \geq 2$.
\end{corollary}
\begin{proof}
This follows easily from the description of $\pi_1(\tau_d)$  and Corollary \ref{htpygps}. 
\end{proof}

\begin{remark} ($d \to \infty$) The canonical inclusion $B\Sigma_d \to B\Sigma_{d+1}$ induces an infinite loop map $j_d: Q(B\Sigma_d)_+ \to Q(B\Sigma_{d+1})_+$, and it is immediate that the maps $\overline{\tau}_d$ and $\overline{\tau}_{d+1} j_d$ agree. Moreover, there is a symmetric monoidal functor $\overline{\mathbb{A}}_d \to \overline{\mathbb{A}}_{d+1}$ that is defined on objects by $\mathbf{n} \mapsto \mathbf{n}$. Together they give a map of homotopy fiber sequences
\begin{displaymath}
 \xymatrix{
Q_0 (B \Sigma_d ){}_+ \ar[r]^{\overline{\tau}_d} \ar[d] & Q_0 S^0 \ar[r] \ar[d] & B \overline{\mathbb{A}}_d \ar[d] \\ 
Q_0 B(\Sigma_{d+1}){}_+ \ar[r]^{\overline{\tau}_{d+1}} & Q_0 S^0 \ar[r]  & B \overline{\mathbb{A}}_{d+1}.
}
\end{displaymath}
By letting $d \to \infty$, we obtain a homotopy fiber sequence $Q_0 Q_0 S^0 \to Q_0 S^0 \to B \overline{\mathbb{A}}_{\infty}$ where the first map is the canonical map induced by the identity 
$1: Q_0 S^0 \to Q_0 S^0$.
\end{remark}

\begin{remark} (representations of $\mathbb{A}_d$) \label{cob3.3} Let $(\mathcal{E}, \otimes, 1)$ be a strict symmetric monoidal category. A symmetric monoidal functor $F: \mathbb{A}_d \to \mathcal{E}$ is uniquely determined by a choice of an object $E \in \Ob\mathcal{E}$ and a morphism $\lambda: E^{\otimes d} \to 1$ which is invariant under the action of the symmetric group on $E^{\otimes d}$. Furthermore, the category of symmetric monoidal representations of $\mathbb{A}_d$ in $\mathcal{E}$ is equivalent to a category of such pairs.
\end{remark}

\section{The Cobordism Categories $\mathcal{C}_{1, \partial}$ and $\mathcal{C}^+_{1, \partial}$} \label{cob4}

\subsection{The cobordism categories $\mathcal{C}_d$ and $\mathcal{C}_{d, \partial}$.} \label{cob4.1}

Let $\mathcal{C}_d$ denote the embedded $d$-dimensional cobordism category of \cite{GMTW}. Broadly speaking, the space of objects is the space of closed smooth $(d-1)$-dimensional submanifolds of high Euclidean space and the space of morphisms is the space of smooth $d$-dimensional embedded cobordisms with collared boundary. For the precise definition of this category and its variations, see \cite{GMTW}. 

 The disjoint union of cobordisms makes $\mathcal{C}_d$ into a symmetric monoidal category. Note that $\pi_0 B \mathcal{C}_d$ is exactly the unoriented bordism group of $(d-1)$-dimensional manifolds. By well-known results in infinite loop space theory (see \cite{Se3}, \cite{May2}), it follows that the classifying space $B\mathcal{C}_d$ is an infinite loop space. The main result of \cite{GMTW} identifies this infinite loop space with the infinite loop space associated to a Thom spectrum. 

Let $Gr_{d}(\R^{d+n})$ denote the Grassmannian manifold of $d$-dimensional linear subspaces in $\R^{d+n}$. Let $\gamma_{d,n}: U_{d,n} \to Gr_d(\R^{d+n})$ be the $d$-dimensional tautological bundle and $\gamma_{d,n}^{\perp}: U_{d,n}^{\perp} \to Gr_d(\R^{d+n})$ be its $n$-dimensional orthogonal complement, i.e., 
\begin{displaymath}
U_{d,n}^{\perp} = \{(V,u) | V \in Gr_{d}(\R^{d+n}), u \in V^{\perp}\}.
\end{displaymath}
Let $MT(d)$ be the Thom spectrum associated with the collection of bundles $\{\gamma_{d,n}^{\perp}\}_n$, i.e., the stable vector bundle that is inverse to the universal $d$-dimensional bundle $\gamma_{d, \infty}$. In detail, the pullback of the bundle $\gamma_{d,n+1}^{\perp}$ along the inclusion $Gr_{d}(\R^{d+n}) \to Gr_{d}(\R^{d + n +1})$ is isomorphic to  $\gamma_{d, n}^{\perp} \oplus \epsilon^1$, so there are maps of Thom spaces 
\begin{displaymath}
S^1 \wedge Th(\gamma_{d,n}^{\perp}) \cong Th(\gamma_{d,n} \oplus \epsilon^1) \to Th(\gamma_{d,n+1}^{\perp}) 
\end{displaymath}
which define the structure maps of a spectrum. Following standard conventions (see \cite[IV.5]{Ru}), the $(d+n)$-th space of $MT(d)$ is the Thom space $Th(\gamma_{d,n}^{\perp})$. 

The isomorphism of bundles $\gamma_{d,n}^{\perp} \stackrel{\cong}{\to} \gamma_{n,d}$ covering the homeorphism $Gr_d(\R^{d+n}) \stackrel{\cong}{\to} Gr_n(\R^{n+d})$, $U \mapsto U^{\perp}$, gives an identification $Th(\gamma_{d,n}^{\perp}) \cong Th(\gamma_{n,d})$.  The inclusions $\gamma_{d,n}^{\perp} \to \gamma_{d+1, n}^{\perp}$ give inclusion maps of spectra $MT(d) \rightarrow \Sigma MT(d+1)$ which define a filtration of the universal Thom spectrum,
\begin{equation*}
MT(0) \rightarrow \Sigma MT(1) \rightarrow \Sigma^2 MT(2) \rightarrow ... \rightarrow MO.
\end{equation*}
In \cite[Proposition 3.1]{GMTW}, it is shown that there is a cofiber sequence of spectra, 
\begin{equation} \label{MTO-cof}
MT(d) \to \Sigma^{\infty} BO(d)_+ \to MT(d-1)
\end{equation}
that induces a homotopy fiber sequence of the associated infinite loop spaces,
\begin{equation} 
 \Omega^{\infty} MT(d) \to QBO(d)_+ \to \Omega^{\infty} MT(d-1).
\end{equation}
For $d=1$, this is the homotopy fiber sequence
\begin{equation}  \label{g-spectra}
\Omega^{\infty} MT(1) \to QBO(1)_+ \simeq Q\R P^{\infty}_+ \to \Omega^{\infty} MT(0) \simeq QS^0
\end{equation}
where the last map is the stable transfer of the universal double covering (e.g. see \cite[Remark 3.2]{GMTW}, \cite[Lemma 2.1]{Gal}).

\begin{theorem}[Galatius-Madsen-Tillmann-Weiss \cite{GMTW}]
There is a weak homotopy equivalence $\alpha: B\mathcal{C}_d \stackrel{\simeq }{\to} \Omega^{\infty-1} MT(d)$ for all $d \geq 0$.
\end{theorem}

The positive boundary cobordism category $\mathcal{C}_{d, \partial}$ is the subcategory of $\mathcal{C}_d$ that contains only the cobordisms $W$ such that the inclusion of the incoming boundary is surjective on path components. In other words, every connected component of a cobordism $W$ in $\mathcal{C}_{d, \partial}$ has a non-empty incoming boundary. Note that here we use the opposite of the convention used in \cite{GMTW} for this technical condition (\textit{negative boundary}?). The introduction of this subcategory is motivated by the $2$-dimensional case and the connection between the embedded cobordism category of $orientable$ $2$-cobordisms $\mathcal{C}_{2,\partial}^+$ and the stable mapping class group \cite{GMTW}, \cite{Ti}.

\begin{theorem}[Galatius-Madsen-Tillmann-Weiss \cite{GMTW}] \label{general positive boundary}
The inclusion map $$B\mathcal{C}_{d,\partial} \rightarrow B\mathcal{C}_d$$ is a weak homotopy equivalence for all $d \geq 2$.
\end{theorem}

The purpose of this section is to show that this is true also in the case $d=1$.

\subsection{$\mathcal{C}_{1, \partial}$ and $\mathbb{A}_2$} \label{cob4.2}

We may regard $\mathbb{A}_2$ as a cobordism category of 1-dimensional cobordisms by thinking of its objects as collections of points and its morphisms as 1-dimensional cobordisms each of whose components has non-empty incoming boundary. Then  $\mathbb{A}_2$ is essentially the category of components of $\mathcal{C}_{1, \partial}$ and we claim that the projection functor is a weak equivalence. Roughly speaking, the reason is because the relevant diffeomorphism groups are homotopically trivial in this dimension.

Consider an auxiliary double category $\mathcal{C}^{\square}_{1,\partial}$:
\begin{itemize}
\item an $object$ is an ordered configuration of $n$ points in $\R^{\infty}$ 
together with a real number that specifies a slice $\{a\} \times \R^{\infty} \subseteq \R \times \R^{\infty}$,
\item a $horizontal$ $morphism$ is a parametrised embedded 1-dimensional cobordism $W$ in $[a,b] \times \R^{\infty}$ with collared boundary such that $\partial W = W \cap (\{a, b\} \times \R^{\infty})$ and each of whose components has non-empty incoming boundary,
\item a $vertical$ $morphism$ is a permutation of a finite set of points,
\item a $square$ is a diffeomorphism of the embedded cobordisms that restricts to the corresponding product diffeomorphism between the collared boundaries.
\end{itemize}

\begin{proposition} \label{double category}
There are weak homotopy equivalences 
\begin{displaymath}
B\mathcal{C}_{1,\partial} \stackrel{\simeq}{\leftarrow} B\mathcal{C}^{\square}_{1,\partial} \stackrel{\simeq}{\rightarrow} B\mathbb{A}_2.
\end{displaymath}
\end{proposition}
\begin{proof}
The nerve of the double category $\mathcal{C}^{\square}_{1,\partial}$ is a bisimplicial space whose $(p, q)$ simplices are $p \times q$ blocks of squares in $\mathcal{C}^{\square}_{1,\partial}$. If we realize with respect to the vertical direction $p$ first, we obtain the simplicial space of the nerve of $\mathcal{C}_{1,\partial}$. Therefore there is a weak homotopy equivalence $B \mathcal{C}^{\square}_{1,\partial} \stackrel{\simeq}{\to} B\mathcal{C}_{1,\partial}$. Let $\Delta(\mathcal{C}^{\square}_{1,\partial})$ be the diagonal simplicial space given by $p \mapsto N_{pp} \mathcal{C}^{\square}_{1,\partial}$. There is a simplicial map $\mathcal{L}: \Delta(\mathcal{C}^{\square}_{1,\partial}) \to N\mathbb{A}_2$ defined as follows: on 0-simplices, it just forgets the embedding, on 1-simplices it takes a square 
\begin{displaymath}
\xymatrix{
. \ar[r]^W \ar[d]^{\sigma_0} & . \ar[d]^{\sigma_1} \\
. \ar[r]^{W'} & . 
}
\end{displaymath}
to the morphism in $\mathbb{A}_2$ that is given by the composition of the permutation $\sigma_0$ followed by the cobordism $W'$, i.e., essentially it forgets the embedding, and so on. It is not difficult to see that this process defines a simplicial map. Furthermore, $\mathcal{L}$ is a weak homotopy equivalence levelwise because the space of diffeomorphisms of a line segment that fixes the boundary points is contractible. Therefore $|\mathcal{L}|$ is also a weak homotopy equivalence, as required.
\end{proof}

\begin{proposition} There is a weak homotopy equivalence of infinite loop spaces $B \mathbb{A}_2 \stackrel{\simeq}{\to} \Omega^{\infty-1} MT(1)$.
\end{proposition}
\begin{proof}
This follows from a comparison between the homotopy fiber sequence \eqref{g-spectra} above and the homotopy fiber sequence of Theorem \ref{mainA} for $d=2$  by the identification of the map $\tau_2$ in Proposition \ref{tr}.
\end{proof} 

\begin{corollary} \label{1 positive boundary}
The inclusion  $B\mathcal{C}_{1,\partial} \stackrel{\simeq}{\to} B\mathcal{C}_1$ is a weak homotopy equivalence of infinite loop spaces.
\end{corollary}
\begin{proof}
This follows from the homotopy commutative diagram
\begin{displaymath}
\xymatrix{
Q\R P^{\infty}_+  \ar[r] & QS^0  \ar[r] & \Omega^{\infty - 1}MT(1) & 
B\mathcal{C}_1 \ar[l]_(.35){\simeq}  \\
\Z \times {B({\Sigma}_{\infty} \wr {\Sigma}_2})^+  \ar[u]^{\simeq} \ar[r] & \Z \times {B{\Sigma}_{\infty}}^+ \ar[r] \ar[u]^{\simeq} &  B\mathbb{A}_2 \ar[u]^{\simeq}  \ar[r]^{\simeq} & 
B\mathcal{C}_{1,\partial}  \ar[u]
}
\end{displaymath}
\end{proof}

\subsection{$\mathcal{C}^+_{1,\partial}$ and $\mathbb{A}_2^{\pm}$} \label{cob4.3}

Let $\mathcal{C}^+_{1, \partial}$ be the positive boundary subcategory of the embedded cobordism category $\mathcal{C}^+_1$ of $oriented$ 1-dimensional cobordisms (see \cite{GMTW}). Analogously to the unoriented case, there are Thom spectra $MT(d)^+$ associated with the complement of the universal bundle of oriented $d$-dimensional bundles. The main result of \cite{GMTW} in the oriented $1$-dimensional case says that there is a weak homotopy equivalence 
\begin{displaymath}
B\mathcal{C}^+_1 \stackrel{\simeq}{\to} \Omega^{\infty - 1} MT(1)^+.
\end{displaymath}
Furthermore, there is an analogous cofiber sequence of spectra
\begin{displaymath}
MT(1)^+ \to \Sigma^{\infty} BSO(1)_+ \simeq \Sigma^{\infty} S^0 \to MT(0)^+ \simeq \Sigma^{\infty} S^0 \vee \Sigma^{\infty} S^0
\end{displaymath}
that induces a homotopy fiber sequence of infinite loop spaces
\begin{equation*} \label{or-g-spectra}
QS^0  \stackrel{\Delta}{\rightarrow} QS^0 \times QS^0 \to \Omega^{\infty -1} MT(1)^+
\end{equation*}
and the first map is the diagonal map (see \cite[Remark 3.2]{GMTW}).

The cobordism category $\mathbb{A}_2^{\pm}$ is an oriented version of $\mathbb{A}_2$ where the points are oriented, i.e. labelled by ``$+$'' or ``$-$''. More precisely, the set of objects is the collection of pairs $(m,n)$ of non-negative integers that we think of as the set of positively oriented and negatively oriented points respectively. A morphism from $(m,n)$ to $(m',n')$ is a cobordism in $\mathbb{A}_2$ from $m+n$ to $m'+n'$ that respects the orientation of the points. More specifically, a morphism is a 1-dimensional cobordism each of whose components is a unit interval that connects points of different orientation if and only if they are both in the source of the morphism. 

$\mathbb{A}_2 ^{\pm}$ is a symmetric monoidal category under the disjoint union of cobordisms and $\pi_0 B\mathbb{A}_2 ^{\pm} \cong \Z$, so $B \mathbb{A}_2 ^{\pm}$ is an infinite loop space.

\begin{theorem} \label{oriented version}
There is a homotopy fiber sequence of infinite loop spaces
\begin{displaymath}
 QS^0 \stackrel{\Delta}{\rightarrow} QS^0 \times QS^0 \to B\mathbb{A}_2^{\pm}.
\end{displaymath}
\end{theorem}
\begin{proof}
The method of the proof is the same as for Theorem \ref{mainA}. The unique morphism $\alpha: (1,1) \to (0,0)$ defines an $H\Z$-stable diagram in the component $\mathbb{A}^{\pm}_{2,(0,0)}$ of $\mathbb{A}^{\pm}_2$ that contains the object $(0,0)$. The little comma categories are equivalent with the symmetric groups, i.e., there are equivalences 
$$ \Sigma_n \stackrel{\sim}{\to} (n+k, n+k) \downarrow_{\mathbb{A}^{\pm}_{2,(0,0)}} (k,k)$$
that respect the stabilisation defined by $\alpha$. Thus the homotopy colimit of the little comma categories $B(\int_{\mathcal{I}(\alpha)} - \downarrow_{\mathbb{A}_{2, (0,0)}^{\pm}} (k,k))$ can be identified canonically with $B\Sigma_{\infty}$ for every $(k,k) \in \mathbb{A}_{2,(0,0)}^{\pm}$. The monoid of endomorphisms of $(n,n)$ is isomorphic with the group $\Sigma_n \times \Sigma_n$ and the forgetful functor 
$ (n,n) \downarrow_{\mathbb{A}_{2, (0,0)}^{\pm}} (0,0) \to \End((n,n))$ can be identified with the diagonal group homomorphism $\Sigma_n \to \Sigma_n \times \Sigma_n$. By Theorem \ref{ggct}, there is a homology fiber sequence
\begin{displaymath}
B\Sigma_{\infty} \to B\Sigma_{\infty} \times B\Sigma_{\infty} \to B\mathbb{A}^{\pm}_{2,(0,0)}
\end{displaymath}
and the first map is induced by the diagonal homomorphisms $\Sigma_n \to \Sigma_n \times \Sigma_n$. The conditions of Proposition \ref{+} are satisfied, so we obtain a homotopy fiber sequence $Q_0 S^0 \stackrel{\Delta}{\rightarrow} Q_0S^0 \times Q_0S^0 \to B\mathbb{A}^{\pm}_{2,(0,0)}$. Finally, the components work out to produce the required homotopy fiber sequence. 
 \end{proof}

\begin{corollary}
There is a weak homotopy equivalence of infinite loop spaces $B\mathbb{A}_2^{\pm} \stackrel{\simeq}{\to} QS^0$.
\end{corollary}

\begin{proposition}
There is a weak homotopy equivalence $ B\mathbb{A}_2^{\pm} \stackrel{\simeq}{\to} B\mathcal{C}^+_{1,\partial}$.
\end{proposition}
\begin{proof}
Similarly with Proposition \ref{double category}.
\end{proof}

\begin{corollary} \label{orposbd}
The inclusion  $B\mathcal{C}^+_{1,\partial} \stackrel{\simeq}{\to} B\mathcal{C}^+_1$ is a weak homotopy equivalence of infinite loop spaces.
\end{corollary}
\begin{proof}
This follows from Theorem \ref{oriented version} similarly with Corollary \ref{1 positive boundary}.
\end{proof}

\section{The Cobordism Categories $\mathcal{A}_{g,d}$} \label{cob5}

\subsection{Definition of $\mathcal{A}_{g,d}$} \label{cob5.1} Fix integers $d \geq 1$ and $g \geq 0$. Following the construction of \cite{Ti2}, the cobordism category $\mathcal{A}_{g,d}$ is a model for a subcategory $\mathcal{M}_{g,d}$ of the cobordism category of Riemann surfaces $\mathcal{M}$ \cite{Se2}. It has the same objects as $\mathcal{M}$ and a morphism from $m$ to $k$ is a Riemann surface with $m$ incoming and $k$ outgoing parametrised boundary circles such that each of its connected components is
\begin{itemize}
\item[(A)] either a complex annulus with one incoming and one outgoing boundary component that represents a morphism $1 \to 1$,
\item[(B)] or a Riemann surface of genus $g$ with $d$ incoming parametrised boundary components that represents a morphism $d \to 0$.
\end{itemize}
Composition of morphisms is defined by gluing Riemann surfaces along their boundary components using the parametrisations. The space of objects is discrete and the topology on the morphism sets is induced from the associated moduli spaces of Riemann surfaces.  

Let $\tilde{\mathcal{A}}_{g,d}$ be the $2$-category that enriches $\mathbb{A}_d$ with $2$-morphisms given by isotopy classes of orientation-preserving diffeomorphisms of surfaces with type (A) or (B). In detail, the objects of $\tilde{\mathcal{A}}_{g,d}$ are the same as in $\mathbb{A}_d$, i.e. the non-negative integers,  the $1$-morphisms are given by orientable smooth $2$-dimensional cobordisms each of whose components is diffeomorphic to (A) or (B) as above, and the $2$-morphisms are given by isotopy classes of orientation-preserving diffeomorphisms of surfaces that fix the boundary pointwise. The disjoint union of cobordisms (and diffeomorphisms) defines a symmetric monoidal pairing on the $2$-category $\tilde{\mathcal{A}}_{g,d}$. (For a more detailed discussion of the construction of such a $2$-category, see \cite{Ti2}.). 

We recall some standard notation for the mapping class groups. Let $\Sigma^q_{g,p}$ denote an orientable surface of genus $g$ with $p$ boundary components and $q$ marked points. Let $\Diff^+(\Sigma^q_{g,p}, \partial, *)$ be the topological group of orientation-preserving diffeomorphisms that fix the boundary and the marked points pointwise. The mapping class group is the group of connected components $\Gamma^q_{g,p}=\pi_0(\Diff^+(\Sigma^q_{g,p}, \partial, *))$.  By the fundamental results of Earle-Eells\cite{EE} and Earle-Schatz \cite{ES}, the projection $\Diff^+(\Sigma^q_{g,p}, \partial, *) \to \Gamma^q_{g,p}$ is a homotopy equivalence if $3g+2q+5p \geq 5$. Let $\Gamma^{(q)}_{g, (p)}$ denote the group of connected components of the topological group $\Diff^+(\Sigma^q_{g,p}, \{\partial\}, \{\ast \})$ of orientation-preserving diffeomorphisms that are allowed to permute the boundary components and the marked points. There is a short exact sequence of groups
\begin{displaymath}
1 \to \Gamma^q_{g,p} \to \Gamma^{(q)}_{g,(p)} \to \Sigma_q \times \Sigma_p \to 1.
\end{displaymath}

The cobordism category $\mathcal{A}_{g,d}$ is the topological category that is obtained from $\tilde{\mathcal{A}}_{g,d}$ by applying the classifying space functor on the morphism categories. Note that the category of the connected components of $\mathcal{A}_{g,d}$ is precisely the category $\mathbb{A}_d$. Disjoint union of surfaces defines a symmetric monoidal structure on $\mathcal{A}_{g,d}$. Since $\pi_0(B \mathcal{A}_{g,d}) \cong \Z_d$ is a group, it follows that $B\mathcal{A}_{g,d}$ is an infinite loop space, see \cite{Se3}, \cite{May2}. 

\textit{Notation.} We will write $\overline{\mathcal{A}}_{g,d}$ to denote the connected component of $\mathcal{A}_{g,d}$ that contains the object $0$ and $\mathbf{n}$ for the object $n \cdot d$. \\

The special case of the topological category $\mathcal{M}_{0,2}$ and its category of components $\mathbb{A}_2$ was studied by Costello in (an earlier version of)\cite{Cos1}, \cite{Cos2}\footnote{Costello denotes $\mathcal{M}_{0,2}$ by $\emph{A}$ and $\mathbb{A}_2$ by $\emph{B}$.} in connection with the homotopy type of the moduli spaces of stable Riemann surfaces. The topology on $\mathcal{M}_{0,2}$ is determined by the moduli space of parametrised complex annuli. This space is homeomorphic to $(0,1) \times (\Diff^+(S^1) \times \Diff^+(S^1))/S^1 \simeq S^1$, where $(0,1)$ comes from the ratio of the radii of the annulus and $\Diff^+(S^1) \times \Diff^+(S^1)/S^1$ comes from the choice of parametrisations up to rotational symmetry. $\mathcal{M}_{0,2}$ acts on the moduli spaces of Riemann surfaces by adding annuli either (A) to extend a boundary component, or (B) to join  pairs of boundary components. This action defines a functor $\mathbb{M}:\mathcal{M}_{0,2} \rightarrow \mathcal{T}op$ that takes the object $n$ to the moduli space of possibly disconnected Riemann surfaces with $n$ parametrised outgoing boundary components subject to the stability condition that no component can be a sphere with $\leq 2$ boundaries or a torus with no boundaries. On morphisms, the functor is defined by gluing Riemann surfaces. 

On the other hand, the category $\mathbb{A}_2$ acts on the moduli spaces of stable Riemann surfaces as follows. There is a functor $\overline{\mathbb{M}}: \mathbb{A}_2 \rightarrow \mathcal{T}op$ that takes the object $n$ to the moduli space of stable, possibly disconnected, Riemann surfaces with $n$ marked smooth points. On morphisms, $\overline{\mathbb{M}}$ is defined by gluing together pairs of marked points. The operation of gluing pairs of marked points adds new nodes to the stable Riemann surface. Costello asks how to ``approximate'' the functor of the compactified moduli spaces $\overline{\mathbb{M}}$ by the functor $\mathbb{M}$ along the projection $\pi_0: \mathcal{M}_{0,2} \to \mathbb{A}_2$. There is a pull-back functor $\pi_{0}^*: \mathrm{Fun}(\mathbb{A}_2, \mathcal{T}op) \rightarrow \mathrm{Fun}(\mathcal{M}_{0,2}, \mathcal{T}op)$ which has a homotopy left adjoint $L\pi_{0*}$. Costello showed that the left homotopy Kan extension of $\mathbb{M}$ along $\pi_0$ is homotopy equivalent to $\overline{\mathbb{M}}$ as functors. Recall that the (rational) homotopy type of the moduli spaces of Riemann surfaces is given by the homotopy of the corresponding mapping class groups and that the mapping class groups are generated by Dehn twists at embedded annuli. A left (homotopy) Kan extension is a best approximation from the left, so its definition involves a (homotopy) colimit operation. Then, roughly speaking, Costello's claim can be based on the observation that the homotopical effect of the action of the complex annuli of $\mathcal{M}_{0,2}$ on a boundary component of a Riemann surface is the same as that of closing the boundary by adding a disc and thus creating a smooth marked point, while the homotopical effect of the action of the complex annuli on a pair of boundary components is that of killing the contribution of the Dehn twist at the annulus, which is  the same as collapsing the embedded circle in the middle of the annulus and thus creating a node. This result is stated in an earlier version of \cite{Cos1} that appeared on the arXiv, but it seems that a proof of it as stated is not available in print at the moment. A different version of this result by Costello can be found in \cite{Cos2}.    

\subsection{The homotopy fiber sequences} \label{cob5.2} The $2$-categorical structure of $\tilde{\mathcal{A}}_{g,d}$ is determined by the groupoids $\tilde{\mathcal{A}}_{g,d}(1,1)$ and $\tilde{\mathcal{A}}_{g,d}(d,0)$.  The groupoid $\tilde{\mathcal{A}}_{g,d}(1,1)$ is equivalent with the mapping class group $\Gamma_{0,1+1} \cong \Z$ as a category with one object.  $\tilde{\mathcal{A}}_{g,d}(d,0)$ is a connected groupoid which is equivalent with the mapping class group $\Gamma_{g,d}$. The set of objects of $\tilde{\mathcal{A}}_{g,d}(d,0)$ is identified with the set of permutations $\Sigma_d$, i.e., every object corresponds to a different labelling of the circles of some fixed reference surface $\Sigma_{g,d}$. In order to determine the homotopy type of the little comma categories of $\mathcal{A}_{g,d}$, we first need to look at the functor
\begin{displaymath}
\phi: \tilde{\mathcal{A}}_{g,d}(d,d) \times \tilde{\mathcal{A}}_{g,d}(d,0) \to \tilde{\mathcal{A}}_{g,d}(d,0)
\end{displaymath}
that gives the composition in $\tilde{\mathcal{A}}_{g,d}$. 

The objects of $\tilde{\mathcal{A}}_{g,d}(d,d)$ correspond to the different labellings of the boundary circles of the $d$ cylinders, so they are also identified with the set of permutations $\Sigma_d$. The group of automorphisms of any object is isomorphic with the group $\Z ^d$. The composition functor $\tilde{\mathcal{A}}_{g,d}(d,d) \times \tilde{\mathcal{A}}_{g,d}(d,d) \to \tilde{\mathcal{A}}_{g,d}(d,d)$ makes $\tilde{\mathcal{A}}_{g,d}(d,d)$ into a monoidal category. On objects, the monoidal pairing is given by the group multiplication of $\Sigma_d$ and on the morphisms by the group multiplication of $\Sigma_d \wr \Z$. The restriction of the composition map $\phi$ at an object defines an action $\Z ^d \times \Gamma_{g,d} \to \Gamma_{g,d}$. The classifying space $B \tilde{\mathcal{A}}_{g,d}(d,d)$ is the topological group $\Sigma_d \wr B \Z \simeq \Sigma_d \wr SO(2)$. \

\begin{lemma} \label{ses} (a) There is a short exact sequence
\begin{displaymath}
1 \to \Gamma^d_g \to \Gamma^{(d)}_g \to \Sigma_d \to 1,
\end{displaymath}
(b)  there is a central group extension 
\begin{displaymath}
1 \to \Z ^d \to \Gamma_{g,d} \to \Gamma^d_g \to 1.
\end{displaymath}
if $g > 0$ or $d > 2$, and (c) there are homotopy equivalences $\Diff^+(\Sigma^{(2)}_0, \{*\}) \simeq O(2)$ and $\Diff^+(\Sigma^{1}_0, *) \simeq SO(2)$.  
\end{lemma}
\begin{proof}
(a) is clear. (b) Gluing discs along the boudary components of $\Sigma_{g,d}$ and extending diffeomorphisms by the identity defines a map $$\Diff^+(\Sigma_{g,d}, \partial) \to \Diff^+(\Sigma^d_g, \ast).$$ The evaluation of the differential of a diffeomorphism $f \in \Diff^+(\Sigma^d_g, \ast)$ at the $d$ marked points gives a map
$\Diff^+(\Sigma^d_g,\ast) \to SO(2)^d$, and there is a homotopy fibration  
\begin{displaymath}
\Diff^+(\Sigma_{g,d}, \partial) \to \Diff^+(\Sigma^d_g, \ast) \to SO(2)^d.
\end{displaymath}
The required short exact sequence comes from the long exact sequence of homotopy groups since $\Diff^+(\Sigma^d_g,\ast)$ and $\Diff^+(\Sigma_{g,d}, \partial)$ have contractible components when $g > 0$ and $d \geq 1$ or $d > 2$ \cite{EE}, \cite{ES}. The homomorphism $\Z^d \to \Gamma_{g,d}$ is the fiber transport of the homotopy fibration and it maps into the Dehn twists near the boundary components. The short exact sequence is central because Dehn twists around simple closed curves $\alpha, \beta$ commute if the intersection number $i(\alpha, \beta)$ is zero, so the Dehn twists at the boundary components are central relative to a generating set of Dehn twists for the mapping class group $\Gamma_{g,d}$. (c) follows from the well-known homotopy equivalence $SO(3) \stackrel{\simeq}{\hookrightarrow} \Diff^+(S^2)$, see \cite{Sm}.
\end{proof}

\begin{remark} (genus $0$) \label{genus 0} The mapping class group $\Gamma_{0,d}$ is isomorphic with the pure ribbon braid group $PR\beta_{d-1}$. There is an isomorphism $PR\beta_{d-1} \cong P\beta_{d-1} \times \Z^{d-1}$, where $P\beta_{d-1}$ denotes the pure braid group. The factor $\Z^{d-1}$ corresponds to the Dehn twists around $d-1$ circles. The center of the pure braid group $P\beta_{d-1}$ is isomorphic with $\Z$ generated by the braid that corresponds to the Dehn twist around a closed curve homotopic to the boundary of the punctured disc (see \cite{Bir}).
\end{remark}

\begin{lemma} \label{little-comma cat} (a) Let $g>0$ or $d > 2$. There is a homotopy equivalence  
\begin{displaymath}
B(\mathbf{k} \downarrow_{\overline{\mathcal{A}}_{g,d}} \mathbf{k'}) \simeq B(\Sigma_{k-k'} \wr \Gamma^{(d)}_g).
\end{displaymath}

(b) There is a homotopy equivalence $B(\mathbf{k} \downarrow_{\overline{\mathcal{A}}_{0,2}} \mathbf{k'}) \simeq B(\Sigma_{k-k'} \wr O(2))$.

(c) There is a homotopy equivalence $B(\mathbf{k} \downarrow_{\overline{\mathcal{A}}_{0,1}} \mathbf{k'}) \simeq B(\Sigma_{k-k'} \wr SO(2))$.
\end{lemma}
\begin{proof}
(a) As was pointed out in section \ref{cob2.1}, the space $B(\mathbf{k} \downarrow_{\overline{\mathcal{A}}_{g,d}} \mathbf{k'})$ is homotopy equivalent with the homotopy quotient of $\overline{\mathcal{A}}_{g,d}(\mathbf{k},\mathbf{k'})$ by the right action of $\overline{\mathcal{A}}_{g,d}(\mathbf{k},\mathbf{k}) \cong \Sigma_{kd} \wr B \Z$. The action of $(B \Z)^{kd}$ restricts to an action on each component of $\overline{\mathcal{A}}_{g,d}(\mathbf{k}, \mathbf{k'})$, and by Lemma \ref{ses}(b), its homotopy quotient is $(B\Gamma^d_g)^{k-k'}$. The components of $\overline{\mathcal{A}}_{g,d}(\mathbf{k}, \mathbf{k'})$ correspond  to the cosets of $\Sigma_{k-k'} \wr \Sigma_d$ in $\Sigma_{kd}$, and the action of $\Sigma_{kd}$ on the components is the canonical transitive action. Therefore the required homotopy quotient is equivalent with the homotopy quotient of $(B\Gamma^d_g)^{k-k'}$ by $\Sigma_{k-k'} \wr \Sigma_d$, where the action of $\Sigma_d$ on each copy of $B\Gamma^d_g$ is induced by the covering action of $\Sigma_d$ on $\overline{\mathcal{A}}_{g,d}(\mathbf{1},\mathbf{0})$. By Lemma \ref{ses}(a), the homotopy quotient of $(B\Gamma^d_g)^{k-k'}$ by $\Sigma_d^{k-k'}$ is $(B\Gamma^{(d)}_g)^{k-k'}$. Finally, the homotopy quotient of $(B\Gamma^{(d)}_g)^{k-k'}$ by $\Sigma_{k-k'}$ is $E \Sigma_{k-k'} \times_{\Sigma_{k-k'}} (B \Gamma^{(d)}_g)^{k-k'}$, and hence the required result.  

(b) Although the case $(g,d)=(0,2)$ is somewhat special (see Remark \ref{special case} below), the arguments here are analogous. In this case, the action of $\overline{\mathcal{A}}_{0,2}(\mathbf{k}, \mathbf{k})$ on $\overline{\mathcal{A}}_{0,2}(\mathbf{k}, \mathbf{k'})$ is transitive, so the homotopy quotient is equivalent with the classifying space of the stabiliser group of any element. Let $\alpha: \mathbf{k} \to \mathbf{k'}$ be an object in $\mathbf{k} \downarrow_{\overline{\mathcal{A}}_{0,2}} \mathbf{k'}$. The automorphism group of this object is exactly the stabiliser subgroup of the element $\alpha$ with respect to the action of $\overline{\mathcal{A}}_{0,2}(\mathbf{k},\mathbf{k})$. Firstly, we determine this stabiliser subgroup at the $2$-categorical level when $\alpha$ is a single annulus $\mathbf{1} \to \mathbf{0}$. The action of the component of the identity of $\overline{\mathcal{A}}_{0,2}(\mathbf{1},\mathbf{1})$ on $\overline{\mathcal{A}}_{0,2}(\mathbf{1},\mathbf{0})$ can be written at the level of the mapping class groups as follows:
\begin{displaymath}
(\Z \times \Z) \times \Z \to \Z 
\end{displaymath}
\begin{displaymath}
(t_1, t_2), m \mapsto t_1 + m + t_2.
\end{displaymath}
The stabiliser subgroup of this (restricted) action at $m=0$ is the subgroup $\{(n, -n) \in \Z \times \Z: n \in \Z \} \unlhd \Z \times \Z$, so it is isomorphic with $\Z$. The group $\Sigma_2$ of components of $\overline{\mathcal{A}}_{0,2}(\mathbf{1},\mathbf{1})$ permutes the (integral) coordinates, so the stabiliser of the full action is isomorphic with $\Sigma_2 \ltimes B \Z$ as a subgroup of $\End_{\overline{\mathcal{A}}_{0,2}}(\mathbf{1},\mathbf{1})$, and where the semi-direct product is defined with respect to the action induced by $n \mapsto -n$. Similarly, the stabiliser subgroups of the action of $End_{\overline{\mathcal{A}}_{0,2}}(\mathbf{k}, \mathbf{k})$ on $\overline{\mathcal{A}}_{0,2}(\mathbf{k}, \mathbf{k'})$ can be identified with $\Sigma_{k-k'} \wr (\Sigma_2 \ltimes B \Z)$. There is a homotopy equivalence $B \Z \simeq SO(2)$ such that the action $n \mapsto -n$ corresponds to the conjugation by the reflection at the $x-$axis. Finally, note that $O(2) = \Sigma_2 \ltimes SO(2)$ as the only non-trivial extension of $SO(2)$ by $\Sigma_2$. (c) is similar. 
\end{proof}

\begin{remark} ($(g,d)=(0,2)$) \label{special case} The cases $d=1, 2$ and $d > 2$ differ because the components of $\Diff^+(\Sigma^d_0, *)$ are not contractible unless $d > 2$.
\end{remark}

Let $i_k: \End(\mathbf{k}) \rightarrow \End(\mathbf{k+1})$ be the homomorphism that corresponds to the canonical inclusions $\Sigma_{kd} \wr B \Z \rightarrow \Sigma_{(k+1)d} \wr B \Z$ for all $\mathbf{k} \in \Ob\overline{\mathcal{A}}_{g,d}$.  Let $S(g,d)$ be a basepoint of $\overline{\mathcal{A}}_{g,d}(\mathbf{1},\mathbf{0})$, and let $\beta_k: \mathbf{k+1} \rightarrow \mathbf{k}$ denote the morphism $1_{\mathbf{k}} \sqcup S(g,d)$ in $\overline{\mathcal{A}}_{g,d}$. Let $\mathcal{I}$ be the subcategory generated by the morphisms $\beta_k$, $k \geq 0$. Since $\sigma  \beta_k = \beta_k i_k(\sigma)$ for every $\sigma \in \End(\mathbf{k})$, it follows that $\overline{\mathcal{A}}_{g,d}$ stabilizes along $\mathcal{I}$, cf. Example \ref{example2.1}.

\begin{proposition} \label{stab2} $\overline{\mathcal{A}}_{g,d}$ is $H \Z$-stable along $\mathcal{I}$ for all $d \geq 1$ and $g \geq 0$.
\end{proposition}
\begin{proof}  The proof is similar to Proposition \ref{stab}. It suffices to show that for every morphism $u: \mathbf{k+1} \to \mathbf{k}$, the canonical map 
\begin{equation} \label{stab2.1}
 B( \mathbf{n} \downarrow_{\overline{\mathcal{A}}_{g,d}} \mathbf{k+1}) \stackrel{u_*}{\longrightarrow} B( \mathbf{n} \downarrow_{\overline{\mathcal{A}}_{g,d}} \mathbf{k})
\end{equation}
induces a homology equivalence between the colimits over $\mathcal{I}$ as $n \to \infty$. If $(g,d) \neq (0,1), (0,2)$,  then by Lemma  \ref{little-comma cat}(a), this map can be identified up to homotopy with an inclusion   
\begin{equation} \label{stab2.2}
E\Sigma_{n-(k+1)} \times_{\Sigma_{n-(k+1)}} (B\Gamma_g^{(d)})^{n-(k+1)} \to E\Sigma_{n-k} \times_{\Sigma_{n-k}} (B\Gamma_g^{(d)})^{n-k}
\end{equation}
that is determined by a choice of a base-point $* \to B\Gamma_{g,d} \to B\Gamma_g^{(d)}$. This choice of basepoint  depends on the morphism $u: \mathbf{k+1} \to \mathbf{k}$.  However, since $B\Gamma_g^{(d)}$ is path-connected, the homotopy class of the map is independent of this choice. The identification of \eqref{stab2.1} with \eqref{stab2.2} is natural in $\mathbf{n}$ only up to a ``re-labelling'', i.e., conjugation by an element of $\Sigma_{n-k}$. But this conjugation does not change the induced homomorphism on homology. Hence the induced map on homology between the colimit of \eqref{stab2.1} is the same as that of \eqref{stab2.2}, so it is a homology equivalence. The cases $(g,d)= (0,1),(0,2)$ are similar.
\end{proof} 

Combining the last proposition with the methods of Section \ref{cob2.2}, we can now prove the main result of this section. 

\begin{theorem} \label{mainB} There are homotopy fiber sequences of infinite loop spaces,
\begin{equation} \label{mainB1}
QBO(2)_+ \to QBSO(2)_+ \to B\mathcal{A}_{0,2} 
\end{equation}
and
\begin{equation} 
Q(B \Gamma^{(d)}_g)_+  \to QBSO(2)_+ \to B\mathcal{A}_{g,d}
\end{equation}
when $g > 0$ or $d >2$.
\end{theorem}
\begin{proof} By Theorem \ref{ggct} and Proposition \ref{stab2}, there are homology fiber sequences 
\begin{align*}
B(\Sigma_{\infty} \wr O(2)) \to B(\Sigma_{\infty} \wr SO(2)) \to B \overline{\mathcal{A}}_{0,2}, \\
B(\Sigma_{\infty} \wr \Gamma^{(d)}_g ) \to B(\Sigma_{\infty} \wr SO(2)) \to B \overline{\mathcal{A}}_{g,d}.
\end{align*}
where the maps on the right are induced by the inclusion of the endomorphisms $$\Sigma_{nd} \wr B \Z = \End_{\overline{\mathcal{A}}_{g,d}}(\mathbf{n}) \to \overline{\mathcal{A}}_{g,d}$$ for all $n$. By infinite 
loop space theory (e.g. see \cite{May}, \cite{Se4}) and the group completion theorem \cite{McDS}, there are weak homotopy equivalences:
\begin{align*}
\Z \times B( \Sigma_{\infty} \wr O(2))^+ \stackrel{\simeq}{\to} QBO(2)_+ \\
\Z \times B(\Sigma_{\infty} \wr \Gamma_g^{(d)})^+ \stackrel{\simeq}{\to} Q(B\Gamma_g^{(d)})_+.
\end{align*}
The conditions of Proposition \ref{+} are satisfied, so the sequences of infinite loop maps, 
\begin{align*}
B(\Sigma_{\infty} \wr O(2))^+ \to B( \Sigma_{\infty} \wr SO(2))^+ \to B\overline{\mathcal{A}}_{0,2}, \\
B(\Sigma_{\infty} \wr \Gamma_g^{(d)})^+\to B( \Sigma_{\infty} \wr SO(2))^+ \to B\overline{\mathcal{A}}_{g,d}
\end{align*}
define homotopy fiber sequences. These homotopy fiber sequences are the restriction of the required homotopy fiber 
sequences at a single component.
\end{proof}

\begin{remark} ($(g,d)=(0,1)$) The same argument shows that there is also a homotopy fiber sequence $QBSO(2)_+ \stackrel{1}{\to} QBSO(2)_+ \to B \mathcal{A}_{0,1}.$ Indeed the category $\mathcal{A}_{0,1}$ has a terminal object $0$, so the classifying space $B \mathcal{A}_{0,1}$ is contractible.
\end{remark}

Due to the contractibility of the components of the relevant diffeomorphism groups, Theorem \ref{mainB} can be stated more uniformly as saying that there is a homotopy fiber sequence $$Q B \Diff^+(\Sigma^{(d)}_g)_+  \to QBSO(2)_+ \to  B\mathcal{A}_{g,d}$$ for all $g \geq 0$ and $d \geq 1$. In fact, the theorem is more accurately viewed as a statement about diffeomorphism groups. (This fact can be traced already in the proof of Lemma \ref{ses}.). The introduction of the mapping class group is merely a techncal point here that facilitates a concrete construction of the categories $\mathcal{A}_{g,d}$. 

Furthermore the projection functor $\pi_0: \mathcal{A}_{g,d} \to \mathbb{A}_d$ induces a map of homotopy fiber sequences,
\begin{displaymath}
\xymatrix{
Q B \Diff^+(\Sigma^{(d)}_g)_+  \ar[r] \ar[d] & QBSO(2)_+ \ar[r] \ar[d] & B\mathcal{A}_{g,d} \ar[d] \\
Q B \Sigma_d{}_+  \ar[r] & QS^0 \ar[r] & B\mathbb{A}_d
}
\end{displaymath}
where the vertical maps are induced by the maps $\Diff^+(\Sigma^{(d)}_g) \to \Sigma_d$ (cf. Lemma \ref{ses}) and $SO(2) \to \ast$ respectively.

\begin{remark} (the map $B\mathcal{A}_{g,d} \to B \mathcal{S}$) There is an inclusion functor $\mathcal{A}_{g,d} \to \mathcal{S}$ where $\mathcal{S}$ denotes Tillmann's model for the surface category \cite{Ti}, \cite{Ti2} with the opposite convention for the positive boundary condition. Tillmann \cite{Ti} proved that there is a weak homotopy equivalence $\Z \times B \Gamma_{\infty}^+ \stackrel{\simeq}{\to} \Omega B \mathcal{S} $ where $\Gamma_{\infty} : = \co_g \Gamma_{g,1}$ is the stable mapping class group that is obtained by gluing a torus with two boundary components and extending diffeomorphisms by the identity. Then it is easy to see that the induced map
\begin{displaymath}
 B \Gamma_{g,d} \simeq \mathcal{A}_{g,d}(d,0) \to \Omega B \mathcal{A}_{g,d} \to \Omega B \mathcal{S} \simeq \Z \times B \Gamma_{\infty}^+
\end{displaymath}
is up to homotopy given by the canonical map $B \Gamma_{g,d} \to\{g\} \times B \Gamma^+_{\infty, (d-1)+1} \simeq B \Gamma_{\infty}^+$. By Harer's stability theorem, the last map induces isomorphisms in integral homology
in degrees $\ast < \epsilon(g)$ that increases to $\infty$ as $g \to \infty$. 
\end{remark}

There is an interesting variant of the categories $\mathcal{A}_{g,d}$ that allows stabilisation with respect to the genus. For simplicity, let us assume that $d=1$. The cobordism category $\mathcal{A}_{g,1+1}$ is defined similarly with $\mathcal{A}_{g,1}$, but we make the following additional specifications:
\begin{itemize}
\item the surface $\Sigma_{g,d}$ is equipped with an embedded disc $D^2 \subseteq \Sigma_{g,d}$ that does not meet the boundary, and  
\item the orientation-preserving diffeomorphisms are required to fix the disc pointwise. 
\end{itemize}
Then the space $\mathcal{A}_{g,1+1}(1,0)$ is equivalent to $B \Gamma_{g,1+1}$. The same arguments that led up to Theorem \ref{mainB} apply also similarly to show that there is a homotopy fiber sequence $$Q (B \Gamma^{1}_{g,1})_+  \to QBSO(2)_+ \to  B\mathcal{A}_{g,1+1}$$ for all $g \geq 0$. The operation of cutting out the embedded disc and adding a torus with two boundary components defines a symmetric monoidal functor $\mathcal{A}_{g,1+1} \to \mathcal{A}_{g+1,1+1}$. This induces a map of homotopy fiber sequences,
\begin{displaymath}
\xymatrix{
Q (B\Gamma^{1}_{g,1})_+  \ar[r] \ar[d] & QBSO(2)_+ \ar[r] \ar[d] & B\mathcal{A}_{g,1+1} \ar[d] \\
Q (B\Gamma^{1}_{g+1,1})_+  \ar[r] & QBSO(2)_+ \ar[r] & B\mathcal{A}_{g+1,1+1}
}
\end{displaymath}
where the vertical map on the left is induced by the standard stabilisation map $\Gamma^{1}_{g,1} \to \Gamma^{1}_{g+1,1}$. Taking the colimit as $g \to \infty$, we obtain a homotopy fiber sequence 
\begin{equation} \label{1 marked point} 
 Q (B \Gamma^{1}_{\infty,1})_+  \to QBSO(2)_+ \to  B\mathcal{A}_{\infty,1}
\end{equation}
where the first map is induced by the evaluation of the differential of a diffeomorphism at the marked point. By \cite[Theorem 1.1]{BT}, this map splits after plus construction, i.e. there is a homotopy equivalence $$(B \Gamma^{1}_{\infty,1})^+ \simeq B\Gamma_{\infty,1}^+ \times BSO(2),$$ so the homotopy fiber sequence \eqref{1 marked point} splits also. 

\subsection{The homotopy type of $\mathcal{A}_{0,2}$} \label{cob5.3}  The first map in the homotopy fiber sequence of Theorem \ref{mainB} is the unique up to homotopy map of infinite loop spaces 
induced by
\begin{displaymath}
B\Diff^+(\Sigma^{(d)}_g) \stackrel{\simeq}{\to} E\Sigma_d \times_{\Sigma_d} B\Diff^+(\Sigma^d_g) \stackrel{1 \times d}{\longrightarrow} E\Sigma_d \times_{\Sigma_d} BSO(2)^d.
\end{displaymath}
where the map $d$ is induced from the evaluation of the differential of a diffeomorphism at the $d$ marked points, cf. Lemma \ref{ses}. In the special case $(g,d)=(0,2)$, we can identify this as a stable transfer map.

\begin{proposition} \label{BO-transfer}
The map $QBO(2)_+ \to QBSO(2)_+$ in \eqref{mainB1} is the stable transfer map of the double covering (associated with) $BSO(2) \to BO(2)$.
\end{proposition}
\begin{proof}
The map is induced by the forgetful functor $\mathbf{1} \downarrow_{\overline{\mathcal{A}}_{0,2}} \mathbf{0} \to \End_{\overline{\mathcal{A}}_{0,2}}(\mathbf{1})$ which can be identified with the homomorphism
\begin{align*}
\Sigma_2 \ltimes SO(2) \to \Sigma_2 \wr SO(2) \\
(\sigma, \theta) \mapsto (\sigma; \theta,- \theta).
\end{align*}
Here $\theta$ is the angle of a rotation in $SO(2)$. This homomorphism induces the pretransfer $BO(2)=B(\Sigma_2 \ltimes SO(2)) \to E\Sigma_2 \times_{\Sigma_2} BSO(2)^2$ of the double covering $BSO(2) \to BO(2)$, hence the result follows.
\end{proof}

Let $\delta_2: or_2 \to BO(2)$ denote the determinant line bundle, i.e. the line bundle associated with the double covering $p_n: Gr_2 ^+ (\R^{2+ \infty}) \to Gr_2(\R^{2+ \infty})$. Let $-\delta_2: BO(2) \to \Z \times BO$ be the (classifying map of) the stable vector bundle of rank $-1$ that is inverse to $\delta_2$.  This produces a Thom spectrum $\mathbf{Th}(-\delta_2)$ in the following way (see also \cite[IV.5]{Ru}). Let $\{-1\} \times BO(n-1) \hookrightarrow \Z \times BO$ denote the canonical inclusion and $X_n:= (-\delta_2)^{-1}(\{-1\} \times BO(n-1))$. The collection of the spaces $X_n$ defines a exhaustive filtration of $BO(2)$:
\begin{displaymath}
X_1 \subseteq X_2 \subseteq \cdots \subseteq X_n \subseteq \cdots \subseteq BO(2).
\end{displaymath}
Let $\delta_{2,n}^{\perp}:= (- \delta_2 |_{X_n})^*(\gamma_{n-1})$ be the pullback of the universal $(n-1)$-dimensional bundle over $BO(n-1)$. There are canonical isomorphisms $\delta_{2,n+1}^{\perp}|_{X_n} \cong \delta_{2,n}^{\perp} \oplus \epsilon^1$, so we obtain a Thom specturm $\mathbf{Th}(-\delta_2)$ where
\begin{displaymath}
\mathbf{Th}(-\delta_2)_n := Th(\delta_{2,n}^{\perp})
\end{displaymath}
and the structure maps $\Sigma Th(\delta_{2,n}^{\perp}) \to Th(\delta_{2,n+1}^{\perp})$ are induced by the pullback square
\begin{displaymath}
\xymatrix{
\delta_{2,n}^{\perp} \oplus \epsilon^1 \ar[r] \ar[d] & \delta_{2,n+1}^{\perp} \ar[d] \\
X_n \ar[r] & X_{n+1}.
}
\end{displaymath}

\begin{theorem} \label{specialcase2}
There is a weak homotopy equivalence of infinite loop spaces $B \mathcal{A}_{0,2} \stackrel{\simeq}{\to} \Omega^{\infty -1} \mathbf{Th}(- \delta_2)$.
\end{theorem}
\begin{proof}
Let $\delta_{2,n}: = \delta_2 |_{X_n}$ and $p_n: S(\delta_{2,n}) \to X_n$ denote the projection from the associated double covering. The bundles $\delta_{2,n}$ and $\delta_{2,n}^{\perp}$ give a cofiber sequence (see \cite[Lemma 2.1]{Gal}),
\begin{equation} \label{thomcofiber}
Th(\delta_{2,n}^{\perp}) \to Th(\delta_{2,n} \oplus \delta_{2,n}^{\perp}) \cong \Sigma^n X_n{}_+ \stackrel{\partial}{\to} Th(\R \oplus p_n^* \delta_{2,n}^{\perp}).
\end{equation}
where the map $\partial$ is a parametrised Pontryagin-Thom map for the double covering $p_n$. The bundle $\R \oplus p_n^* \delta_{2,n}^{\perp}$ is trivial, so the cofiber sequence \eqref{thomcofiber} can be written as
\begin{displaymath}
\mathbf{Th}(-\delta_2)_n \to \Sigma^n X_n{}_+ \stackrel{\partial}{\to} \Sigma^n S(\delta_{2,n})_+
\end{displaymath}
where the last map induces the stable transfer map of the double covering $p_n$. Since $Gr_2(\R^{2 + n})$ is contained in $X_N$ for some $N$ sufficiently large, it follows that the connectivity of the inclusions $X_n \to BO(2)$ and $S(\delta_{2,n}) \to BSO(2)$ increases to $\infty$ as $n \to \infty$. Therefore there is a cofiber sequence of spectra
\begin{displaymath}
\mathbf{Th}(- \delta_2) \to \Sigma^{\infty} BO(2)_+ \to \Sigma^{\infty} BSO(2)_+ 
\end{displaymath}
and so a homotopy fiber sequence of infinite loop spaces,
\begin{equation*}
\Omega^{\infty} \mathbf{Th}(-\delta_2) \to \Omega^{\infty} \Sigma^{\infty} BO(2)_+ \to \Omega^{\infty} \Sigma^{\infty} BSO(2)_+
\end{equation*}
where the last map is the stable transfer of the double covering $BSO(2) \to BO(2)$. Then the result follows by a comparison with the homotopy fiber sequence of Theorem \ref{mainB} and Proposition \ref{BO-transfer}.
\end{proof}

\renewcommand{\thesubsection}{A.\arabic{subsection}}
\setcounter{subsection}{1}

\appendix
\begin{center}
 {\bf Appendix : Comma categories and $\mathcal{W}$-fiber sequences} 
\end{center}

\subsection{Comma categories} \label{A.1} The purpose of this appendix is to discuss some generalisations of the results of Section \ref{cob2}. Let $\mathcal{C}$ be a small topological category with a discrete space of objects such that the inclusion $\Ob\mathcal{C} \to \Mor\mathcal{C}$ is a cofibration. 

Let $\mathcal{D}$ be a small topological category and $F: \mathcal{D} \to \mathcal{C}$ a continuous functor. For every $c \in \Ob\mathcal{C}$, there is an over-category $F \downarrow c$. The objects are the morphisms $F(d) \rightarrow c$ in $\mathcal{C}$ and a morphism from $F(d) \stackrel{u}{\to} c$ to $ F(d') \stackrel{v}{\to} c$ is given by a morphism $d \stackrel{f}{\to} d'$ such that the triangle
\begin{displaymath}
\xymatrix{
F(d) \ar[r]^{F(f)} \ar[d]^u  & F(d') \ar[dl]^v \\
c & }
\end{displaymath}
commutes in $\mathcal{C}$.  Composition of morphisms in $F \downarrow c$ is defined by the composition in $\mathcal{D}$. This defines a topological category where the topology is induced by the topologies of $\mathcal{C}$ and $\mathcal{D}$. Furthermore, the definition of $F \downarrow c$ is natural in $c$, i.e., for every pair $(\mathcal{D}, F: \mathcal{D} \to \mathcal{C})$, there is a continuous functor 
\begin{align*}
\mathcal{F}(\mathcal{D},F): \mathcal{C} \to \mathcal{C}at(\mathcal{T}op) \\
c \mapsto \mathcal{F}(\mathcal{D},F)(c):= F \downarrow c \\
( c \stackrel{f}{\to} c' ) \mapsto  (F \downarrow c \stackrel{f_*}{\to} F \downarrow c')).
\end{align*}

Let us consider a category $T(\mathcal{C})$ of such pairs. The objects of $T(\mathcal{C})$ are pairs $(\mathcal{D}, F: \mathcal{D} \to \mathcal{C})$ where $\mathcal{D}$ is a small topological category and $F$ is a continuous functor. A morphism from $(\mathcal{D}, F)$ to $(\mathcal{E},G)$ is a pair $(\kappa, \eta)$ where $\kappa: \mathcal{D} \to \mathcal{E}$ is a continuous functor and $\eta: G \circ \kappa \rightarrow F$ is a natural transformation. Note that the definition of $T(\mathcal{C})$ is natural in $\mathcal{C}$. 

\begin{example} If $\mathcal{D}= End_{\mathcal{C}}(c) \stackrel{F}{\to} \mathcal{C}$ is the natural inclusion, then $\mathcal{F}(\mathcal{D}, F)$ is the diagram of little comma categories $\mathcal{F}_c$ of section \ref{cob2.1}. 
\end{example}

Let $\mathrm{Fun}(\mathcal{C}, \mathcal{C}at(\mathcal{T}op))$ denote the category of continuous functors $\mathcal{C} \to \mathcal{C}at(\mathcal{T}op)$ and natural transformations between them. The following proposition may be considered as a generalisation of the Yoneda lemma.

\begin{proposition} \label{superyoneda}
There is a fully faithful functor $\mathcal{F}: T(\mathcal{C}) \to \mathrm{Fun}(\mathcal{C}, \mathcal{C}at(\mathcal{T}op))$.
\end{proposition}
\begin{proof}  Let us clarify the definition of $\mathcal{F}$ on morphisms: a morphism $(\kappa, \eta): (\mathcal{D}, F) \to (\mathcal{E}, G)$ in $T({\mathcal{C}})$ goes to the natural transformation $\mathcal{F}_{(\kappa,\eta)}: \mathcal{F}(\mathcal{D}, F) \rightarrow \mathcal{F}(\mathcal{E}, G)$ whose component functor at the object $c \in \Ob\mathcal{C}$ is induced by the diagram
\begin{displaymath}
\xymatrix{
G \kappa(d) \ar[r]^{G\kappa(\sigma)} \ar[d]^{\eta_d} & G\kappa(d') \ar[d]^{\eta_{d'}} \\
F(d) \ar[d] \ar[r]^{F(\sigma)} & F(d') \ar[dl] \\
c 
}
\end{displaymath}
It is clear that $\mathcal{F}$ is faithful. In order to see that it is also full, let $\delta: \mathcal{F}(D, F) \rightarrow \mathcal{F}(\mathcal{E}, G)$ be a natural transformation 
between diagrams of comma categories that are induced by $F: \mathcal{D} \to \mathcal{C}$ and $G: \mathcal{E} \to \mathcal{C}$ respectively. For every $u: c \to c'$ in $\mathcal{C}$, there is a commutative diagram
\begin{displaymath}
\xymatrix{
F \downarrow c \ar[r]^{\delta_c} \ar[d]^{u_*} & G \downarrow c \ar[d]^{u_*} \\
F \downarrow c' \ar[r]^{\delta_{c'}} & G \downarrow c'
}
\end{displaymath}
Let $p_c : G \downarrow c \to \mathcal{E}$ denote the obvious forgetful functor. We define a functor $\kappa: \mathcal{D} \to \mathcal{E}$ as follows:
\begin{itemize}
\item on objects: $\kappa(d) = p_{F(d)} \circ \delta_{F(d)}(d, 1_{F(d)})$, and
\item on morphims: $\kappa(d \stackrel{\sigma}{\to} d') = p_{F(d')} \circ \delta_{F(d')}((d, F(\sigma)) \stackrel{\sigma}{\to} (d', 1_{F(d')}))$
\end{itemize}
There is a natural transformation $\eta: G \kappa \to F$ whose component at $d \in \Ob \mathcal{D}$ is given by the morphism $\delta_{F(d)}(d, 1_{F(d)})$ in $\mathcal{C}$. By functoriality, it can be verified that $\delta$ is induced by the pair $(\kappa, \eta)$, so $\mathcal{F}$ is full.
\end{proof}

\begin{remark}
The inclusion of an object $* \to \mathcal{C}$ defines an object in $T(\mathcal{C})$. Moreover, there is a functor ${i}: \mathcal{C}^{op} \to T(\mathcal{C})$ that is defined on objects by the inclusion of the objects of $\mathcal{C}$. The composite $\mathcal{F} \circ i$ is the classical Yoneda embedding. 
\end{remark}

\subsection{Lemma \ref{Y} revisited} We prove the following generalisation of Lemma \ref{Y}. 

\begin{proposition} \label{Y-ext}
There is a natural homotopy equivalence 
\begin{displaymath}
 \xymatrix{
u_{(\mathcal{D},F)}: B (\int_{\mathcal{C}} \mathcal{F}(\mathcal{D},F))  \ar[r]^(0.70){\simeq} & B \mathcal{D} \\
}
\end{displaymath}
 for every $(\mathcal{D},F) \in \Ob T(\mathcal{C})$.
\end{proposition}
\begin{proof} The proof is similar to the (second) proof of Lemma \ref{Y}.  There is a forgetful functor $U: \int_{\mathcal{C}} \mathcal{F}(\mathcal{D},F) \to \mathcal{D}$. There is also a functor $G: \mathcal{D} \to \int_{\mathcal{C}} \mathcal{F}(\mathcal{D},F)$ defined as follows: the object $d$ goes to $(F(d),(d, 1_{F(d)}))$ and a morphism $f$ in $\mathcal{D}$ is sent to $(F(f),F(f))$. Note that $U G = 1$ so that $(BU)(BG) = 1$ as well. The pair $(G,U)$ defines an adjunction, so there is a natural transformation $\rho: G U \rightarrow 1$ whose components are as follows: given $(c,d, \alpha: Fd \to c) \in \Ob(\int_{\mathcal{C}} \mathcal{F}(\mathcal{D},F))$, $\rho_c$ is $(\alpha, 1_d)$, i.e. the diagram
\begin{displaymath}
\xymatrix{
Fd \ar[r]^{1} \ar[d]^{1}  & Fd \ar[d]^{\alpha} \\
Fd \ar[r]^{\alpha} & c 
}
\end{displaymath}
It follows that $(BG)(BU) \simeq 1$, hence result.
 \end{proof}

By Proposition \ref{superyoneda}, a natural transformation $\mu: \mathcal{F}(\mathcal{D},F)  \rightarrow \mathcal{F}(\mathcal{D}',F')$ is given by a morphism $(\kappa, \eta): (\mathcal{D}, F) \to (\mathcal{D}',F')$ in $T(\mathcal{C})$. Let $J$ be a small category and $G: J \to T(\mathcal{C})$ a functor.  Associated to this, there is a functor 
\begin{displaymath}
\mathcal{G}_{J}: \mathcal{C}  \to \mathcal{C}at(\mathcal{T}op)
\end{displaymath}
defined on objects by 
\begin{displaymath}
\xymatrix{
c \ar@{|->}[r] & \int_J \mathcal{F}(G -)(c). \\
}
\end{displaymath}
Here $\mathcal{F}$ is the functor of Proposition \ref{superyoneda}. Let $U:T(\mathcal{C}) \to \mathcal{C}at(\mathcal{T}op)$ denote the forgetful functor that is defined on objects by $(\mathcal{D},F) \mapsto D$. For every $c \in \Ob \mathcal{C}$, there is a canonical functor
\begin{displaymath}
\xymatrix{
\mathcal{G}_J(c) \ar[r] & \int_J UG \\
}
\end{displaymath}
that is induced by the canonical forgetful functors $\mathcal{F}(G(j))(c) \to U(G(j))$.

\subsection{$\mathcal{W}$-fiber sequences} Let $\mathcal{W}$ be a class of morphisms in $\mathcal{T}op$. A sequence of maps 
\begin{displaymath}
\xymatrix{
F \ar[r]^i & E \ar[r]^p & B \\
}
\end{displaymath}
together with a homotopy $H: p  i \simeq \mathrm{const}_b$ is called a $\mathcal{W}$-fiber sequence if the canonical map $F \to \mathrm{hofiber}_b(p):= E \times_B \mathrm{Path}(B)$ is in $\mathcal{W}$. Here $\mathrm{Path}(B)=Map(([0,1], \{0\}),(B,b))$ and $ev_1: \mathrm{Path}(B) \to B$ denotes the path fibration. Following \cite{McDS}, we also define a map $f: E \to B$ to be a local $\mathcal{W}$-fibration if every $b \in B$ has arbitrarily small contractible neighborhoods $U$ such that the inclusion $f^{-1}(b') \to f^{-1}(U)$ is in $\mathcal{W}$ for all $b'$ in $U$. 

We will be interested in classes $\mathcal{W}$ that are closed under passage from the local to the global in the following sense.

\begin{definition}
A class of morphisms $\mathcal{W}$ in $\mathcal{T}op$ is called a \textit{localiser} \footnote{The terminology is inspired by Grothendieck's theory of \textit{localisateurs} in $\mathcal{C}at$, see \cite{Mal}.} if the following are satisfied:
\begin{itemize}
\item[(a)] $\mathcal{W}$ contains the weak homotopy equivalences and satisfies the ``2-out-of-3'' property,
\item[(b)] $\mathcal{W}$ is closed under homotopy pushouts, i.e., given a diagram
\begin{displaymath}
\xymatrix{
X_1 \ar[d]^{\sim} & X_0 \ar[l]_{i_1} \ar[r] \ar[d]^{\sim} & X_2 \ar[d]^{\sim} \\
Y_1 & Y_0 \ar[l]_{i_2} \ar[r] & Y_2 
}
\end{displaymath}
where $i_1$ and $i_2$ are cofibrations and the vertical maps are in $\mathcal{W}$, then $X_1 \cup_{X_0} X_2 \to Y_1 \cup_{Y_0} Y_2$ is also in $\mathcal{W}$,
\item[(c)] $\mathcal{W}$ is closed directed homotopy colimits, i.e., given a diagram
\begin{displaymath}
\xymatrix{
 X_0 \ar[r] \ar[d] & X_1 \ar[r] \ar[d] & \cdots \ar[r] & X_n \ar[r] \ar[d] & \cdots \\
Y_0 \ar[r] & Y_1 \ar[r] & \cdots \ar[r] & Y_n \ar[r] & \cdots
}
\end{displaymath}
where all the horizontal maps are cofibrations and the vertical maps are in $\mathcal{W}$, then the canonical map $ \co_n X_n \to \co_n Y_n$ is in $\mathcal{W}$.
\end{itemize} 
\end{definition}

\begin{example} 
For every generalised homology theory $E_*$ that satisfies the limit axiom, the class of $E_*$-equivalences of spaces defines a localiser on $\mathcal{T}op$. 
\end{example}

Let $\mathcal{W}$ denote an arbitrary localiser in $\mathcal{T}op$ for the rest of this section.

\begin{proposition}
If a map $F: X_. \to Y_.$ of good simplicial spaces is levelwise in $\mathcal{W}$, then so is also its geometric realisation $|F|: |X_.| \to |Y_.|$. 
\end{proposition}
\begin{proof}
The realisation of a good simplicial space can be formed naturally as the directed colimit of iterative pushouts along a cofibration \cite[Appendix A]{Se3}. Hence the result follows from the properties of the localiser $\mathcal{W}$. 
\end{proof}

\begin{proposition} \label{local-global} If $f: E \to B$ is a local $\mathcal{W}$-fibration and $B$ is locally contractible and paracompact, then $f^{-1}(b) \to E \stackrel{f}{\to} B$ is a $\mathcal{W}$-fiber sequence.
\end{proposition}

First, we prove the following key lemma. 

\begin{lemma} \label{contractible diagrams}
Let $\mathcal{I}$ be a small category with weakly contractible nerve and $F: \mathcal{I} \to \mathcal{T}op$ be a diagram of spaces such that $F(f)$ is in $\mathcal{W}$ for every $f \in \Mor \mathcal{I}$. Then the canonical map 
$F(i) \to \ho_{\mathcal{I}} F$ is also in $\mathcal{W}$ for every $i \in \Ob \mathcal{I}$.
\end{lemma}
\begin{proof}
Let $\mathrm{Simp}(\mathcal{I})$ denote the category of simplices associated to the nerve of $\mathcal{I}$. There is a canonical projection functor $\epsilon: \mathrm{Simp}(\mathcal{I}) \to \mathcal{I}$ which is homotopically cofinal, e.g. see \cite{CS}.  It follows that the canonical map $$\ho_{\mathrm{Simp}(\mathcal{I})} \epsilon^* F \to \ho_{\mathcal{I}} F$$ is a weak homotopy equivalence for every $F: \mathcal{I} \to \mathcal{T}op$. This way we reduce to the case of a simplex category. Every weakly contractible simplex category is the colimit of pushouts along the horn inclusion of the simplex category of $\Lambda^k_n$ into the simplex category of $\Delta[n]$. Then the result follows from the closure properties of localisers by similar arguments as in \cite[Lemma 27.8]{CS}.
\end{proof}

\begin{proposition} \label{local-global2}
Let $f: E \to B$ be as above and assume that $B$ is contractible. Then $f^{-1}(b) \to E$ is in $\mathcal{W}$ for every $b \in B$. 
\end{proposition}
\begin{proof}
Let $\mathcal{U}= \{U_{\alpha} \}_{\alpha \in J}$ be a basis of contractible open sets $U$ of $B$ such that $f^{-1}(b) \to f^{-1}(U)$ is in $\mathcal{W}$ for all $b \in U$. Let $\mathcal{I}_{\mathcal{U}}$ denote the poset whose objects are the elements of the index set $J$ and the order is defined by $j \leq j'$ if $U_j \subseteq U_{j'}$. The nerve $N \mathcal{I}_{\mathcal{U}}$ is contractible because it is homotopy equivalent to $B$, e.g. see \cite{Se}. Let $F: \mathcal{I}_{\mathcal{U}} \to \mathcal{T}op$, $j \mapsto f^{-1}(U_j)$, be the functor of the induced numerable open covering of $E$. There is a weak homotopy equivalence $\ho_{\mathcal{I}_{\mathcal{U}}} F \stackrel{\simeq}{\to} E$ (see \cite{Se}). Since $\mathcal{W}$ satisfies the ``2-out-of-3'' property, the functor $F$ takes values in $\mathcal{W}$. 
Hence Lemma \ref{contractible diagrams} implies that the canonical map $F(j) \to \ho_{\mathcal{I}_{\mathcal{U}}} F$ is a weak homotopy equivalence. Then the result follows.  
\end{proof}

\begin{proof}[Proof of Proposition \ref{local-global}]
It follows from Proposition \ref{local-global2} similarly with \cite[Proposition 5]{McDS}. 
\end{proof}

We can now state the generalisation of Theorem \ref{GGCT} for an arbitrary localiser $\mathcal{W}$. We make the assumption that $\Mor \mathcal{C}$ is locally contractible and paracompact from now on.

\begin{theorem} \label{W-fiber-seq1}
If $F: \mathcal{C} \to \mathcal{C}at(\mathcal{T}op)$ is a continuous functor such that $B F(f)$ is in $\mathcal{W}$ for all $f \in \Mor\mathcal{C}$, then 
\begin{displaymath}
\xymatrix{
B(F(c)) \ar[r] &  B(\int_{\mathcal{C}} F) \ar[r] & B\mathcal{C}
}
\end{displaymath}
is a $\mathcal{W}$-fiber sequence. 
\end{theorem}
\begin{proof}
The method of the proof is very familiar \cite{McDS}, \cite{Ti}, \cite{PS}. We are going to proceed inductively on the skeleton of the thick realisation $|| N \mathcal{C} ||$ defined using only the face maps. This has the same homotopy type as $B \mathcal{C}$ \cite[Appendix A]{Se3}. The $n$-skeleton $||N \mathcal{C} ||_n$ is obtained inductively as the homotopy pushout 
\begin{displaymath}
\co( || N \mathcal{C} ||_{n-1} \leftarrow \partial \Delta[n] \times N_n \mathcal{C} \hookrightarrow \Delta[n] \times N_n\mathcal{C} )
\end{displaymath} 
where the left hand-side map is defined using the face maps. For $n=0$, we have the sequence $B(F(c)) \to \Ob(\int_{\mathcal{C}} F) \to \Ob\mathcal{C}$, which is clearly a $\mathcal{W}$-fiber sequence. By the property (b) of localisers, the sequence of maps 
\begin{displaymath}
\xymatrix{
B(F(c)) \ar[r] & ||N(\mathcal{C} \wr  BF )||_n \ar[r] & ||N\mathcal{C}||_n
}
\end{displaymath}
is a local $\mathcal{W}$-fibration, and so also a $\mathcal{W}$-fiber sequence by Proposition \ref{local-global} (compare with \cite[Proposition 3]{McDS} and \cite[Proposition 2.2]{PS}). Then the result follows from the property (c) of localisers (compare with \cite[Proposition 2.3]{PS}) and Proposition \ref{thomason-grothendieck}.
\end{proof}

\begin{remark} \label{technical-assumption}
We note that for the purpose of applying Theorem \ref{W-fiber-seq1}, the additional point-set topological assumptions on $\mathcal{C}$ can be assumed without loss of generality. This is because every $\mathcal{C}$ (as in \ref{A.1}) admits a ``CW-approximation'' from the left, i.e. there is new topological category $\tilde{\mathcal{C}}$ such that $\Mor(\tilde{\mathcal{C}})$ is a CW-complex (and satisfies the assumptions of \ref{A.1}) together with a continuous functor $Q: \tilde{\mathcal{C}} \to \mathcal{C}$ that induces a weak homotopy equivalence between all morphism spaces. The construction of $\tilde{\mathcal{C}}$ is done by applying the singular chains functor $S_*: \mathcal{T}op \to s\mathcal{S}et$ followed by the geometric realisation $|-|: s\mathcal{S}et \to \mathcal{T}op$ on the individual morphism spaces of $\mathcal{C}$. Since both functors preserve pullbacks, this produces a well-defined topological category $\tilde{\mathcal{C}}$ together with a continuous functor $Q: \tilde{\mathcal{C}} \to \mathcal{C}$ induced by the counit map $|-| \circ S_* \to 1_{\mathcal{T}op}$ which is a pointwise weak homotopy equivalence of functors.  
\end{remark}

The last theorem has two important consequences. The first one is a generalisation of Quillen's Theorem B \cite{Q}.

\begin{theorem}
Let $F: \mathcal{D} \to \mathcal{C}$ be an object in $\mathcal{T}(\mathcal{C})$. If $B(u_*): B(\mathcal{D} \downarrow c) \to B(\mathcal{D} \downarrow c')$ is in $\mathcal{W}$ for every $u: c \to c'$ in $\mathcal{C}$, then $B(\mathcal{D} \downarrow c) \to B \mathcal{D} \to B \mathcal{C}$ is a $\mathcal{W}$-fiber sequence.
\end{theorem}
\begin{proof}
The diagram $\mathcal{F}(\mathcal{D}, F): \mathcal{C} \to \mathcal{C}at(\mathcal{T}op)$ satisfies the conditions of Theorem \ref{W-fiber-seq1}, so the sequence 
\begin{displaymath}
\xymatrix{
B (\mathcal{D} \downarrow c) \ar[r] & B( \int_{\mathcal{C}} \mathcal{F}(\mathcal{D}, F)) \ar[r] & B \mathcal{C}
}
\end{displaymath}
is a $\mathcal{W}$-fiber sequence. By Lemma \ref{Y-ext}, there is a canonical homotopy equivalence $B( \int_{\mathcal{C}} \mathcal{F}(\mathcal{D}, F)) \stackrel{\simeq}{\to} B \mathcal{D}$, hence the result follows. 
\end{proof}

We also have a generalisation of Theorem \ref{ggct} for an arbitrary localiser $\mathcal{W}$.  

\begin{definition}  A diagram $G: J \to T(\mathcal{C})$ is called $\mathcal{W}$-stable if for every morphism $u: c \to c'$ in $\mathcal{C}$, the map $B \mathcal{G}_{J}(c) \to B\mathcal{G}_{J}(c')$ is in $\mathcal{W}$. 
\end{definition}

\begin{theorem} \label{W-fiber-seq2}
If $G: J \to T(\mathcal{C})$ is a $\mathcal{W}$-stable diagram, then there is $\mathcal{W}$-fiber sequence
\begin{displaymath}
\xymatrix{
B\mathcal{G}_{J}(c) \to B \int_J (U G) \to B\mathcal{C}
}
\end{displaymath}
for every $c \in \Ob \mathcal{C}$.
\end{theorem}
\begin{proof}
This follows from Theorem \ref{W-fiber-seq1} and Proposition \ref{Y-ext} with the same arguments as in the proof of Theorem \ref{ggct}.
\end{proof}

\section*{Acknowledgements}

This work is a part of my D.Phil. thesis. I am grateful to my D.Phil. supervisor, Prof. Ulrike Tillmann, for her valuable advice and support. I also wish to gratefully acknowledge the financial support of an EPSRC studentship 
and a scholarship from the Onassis Public Benefit Foundation.

\end{document}